\documentclass[letterpaper,11pt]{article}

\renewcommand{\subsubsection}{\paragraph}

\usepackage[margin=1in]{geometry}  
\usepackage{xspace,exscale,relsize}
\usepackage{fancybox,shadow}
\usepackage{graphicx}
\usepackage{color}
\usepackage{amsthm,amsfonts}
\usepackage{amssymb}
\usepackage{authblk}
\usepackage[title]{appendix}

\usepackage{extarrows}

\usepackage{enumerate}

\usepackage[noadjust]{cite}

\usepackage{amsmath}
	\makeatletter
	\let\over=\@@over \let\overwithdelims=\@@overwithdelims
	\let\atop=\@@atop \let\atopwithdelims=\@@atopwithdelims
  	\let\above=\@@above \let\abovewithdelims=\@@abovewithdelims
  	\makeatother
\usepackage{rotating}

\usepackage{ifpdf}

\usepackage{subfigure}
\usepackage{psfrag}

\usepackage{prettyref}

\usepackage{tikz}
\usetikzlibrary{arrows}
\tikzstyle{int}=[draw, fill=blue!20, minimum size=2em]
\tikzstyle{dot}=[circle, draw, fill=blue!20, minimum size=2em]
\tikzstyle{init} = [pin edge={to-,thin,black}]

\usepackage[
            CJKbookmarks=true,
            bookmarksnumbered=true,
            bookmarksopen=true,
            colorlinks=true,
            citecolor=red,
            linkcolor=blue,
            anchorcolor=red,
            urlcolor=blue,
            pdfauthor={Polyanskiy and Wu}
            ]{hyperref}

\usepackage[all]{xy}

\usepackage{mathtools}

\usepackage{ifthen}
\newboolean{aos}
\setboolean{aos}{FALSE}

\newcommand{\mreals}{\ensuremath{\mathbb{R}}}

\newcommand{\supp}{\ensuremath{\mathrm{supp}}}

\ifx\eqref\undefined
	\newcommand{\eqref}[1]{~(\ref{#1})}
\fi
\ifx\mod\undefined
	\def\mod{\mathop{\rm mod}}
\fi

\usepackage{bm}

\def\argmax{\mathop{\rm argmax}}

\def\exp{\mathop{\rm exp}}

\def\Var{\mathrm{Var}}

\def\PP{\mathbb{P}}

\def\eqdef{\triangleq}

\newcommand{\pen}{\textrm{pen}}
\newcommand{\Lopt}{L_{\rm opt}}
\newcommand{\Exp}{\mathrm{Exp}}
\newcommand{\Law}{\mathrm{Law}}
\newcommand{\utheta}{\underline{\theta}}
\newcommand{\otheta}{\overline{\theta}}

\newcommand{\SG}{\mathsf{SG}}

\newcommand{\xmin}{x_{\min}}
\newcommand{\xmax}{x_{\max}}
\newcommand{\thetamin}{\theta_{\min}}
\newcommand{\thetamax}{\theta_{\max}}
\newcommand{\NPMLE}{\mathrm{NPMLE}}
\newcommand{\NPMLEpi}{\hat{\pi}_{\NPMLE}}

\newcommand{\Poi}{\mathrm{Poi}}

\newcommand{\reals}{\mathbb{R}}
\newcommand{\naturals}{\mathbb{N}}
\newcommand{\integers}{\mathbb{Z}}
\newcommand{\complex}{\mathbb{C}}
\newcommand{\Expect}{\mathbb{E}}

\newcommand{\Prob}{\mathbb{P}}
\newcommand{\prob}[1]{\mathbb{P}\left[#1\right]}

\newcommand{\TV}{{\rm TV}}

\newcommand{\iid}{i.i.d.\xspace}

\newcommand{\pth}[1]{\left( #1 \right)}

\newcommand{\sth}[1]{\left\{ #1 \right\}}

\newcommand{\iiddistr}{{\stackrel{\text{\iid}}{\sim}}}

\newcommand{\indc}[1]{{\mathbf{1}_{\left\{{#1}\right\}}}}

\definecolor{myblue}{rgb}{.8, .8, 1}
\definecolor{mathblue}{rgb}{0.2472, 0.24, 0.6} 
\definecolor{mathred}{rgb}{0.6, 0.24, 0.442893}
\definecolor{mathyellow}{rgb}{0.6, 0.547014, 0.24}

\newcommand{\calI}{{\mathcal{I}}}

\newcommand{\calM}{{\mathcal{M}}}

\newcommand{\calP}{{\mathcal{P}}}

\newcommand{\calX}{{\mathcal{X}}}

\newcommand{\diverge}{\to \infty}

\newrefformat{eq}{(\ref{#1})}
\newrefformat{thm}{Theorem~\ref{#1}}
\newrefformat{th}{Theorem~\ref{#1}}
\newrefformat{chap}{Chapter~\ref{#1}}
\newrefformat{sec}{Section~\ref{#1}}
\newrefformat{seca}{Section~\ref{#1}}
\newrefformat{algo}{Algorithm~\ref{#1}}
\newrefformat{fig}{Fig.~\ref{#1}}
\newrefformat{tab}{Table~\ref{#1}}
\newrefformat{rmk}{Remark~\ref{#1}}
\newrefformat{clm}{Claim~\ref{#1}}
\newrefformat{def}{Definition~\ref{#1}}
\newrefformat{cor}{Corollary~\ref{#1}}
\newrefformat{lmm}{Lemma~\ref{#1}}
\newrefformat{prop}{Proposition~\ref{#1}}
\newrefformat{pr}{Proposition~\ref{#1}}
\newrefformat{app}{Appendix~\ref{#1}}
\newrefformat{apx}{Appendix~\ref{#1}}
\newrefformat{ex}{Example~\ref{#1}}
\newrefformat{exer}{Exercise~\ref{#1}}
\newrefformat{soln}{Solution~\ref{#1}}
\newrefformat{ass}{Assumption~(\ref{#1})}

\def\unifto{\mathop{{\mskip 3mu plus 2mu minus 1mu%
	\setbox0=\hbox{$\mathchar"3221$}%
	\raise.6ex\copy0\kern-\wd0%
	\lower0.5ex\hbox{$\mathchar"3221$}}\mskip 3mu plus 2mu minus 1mu}}

\ifx\lesssim\undefined
\def\simleq{{{\mskip 3mu plus 2mu minus 1mu%
	\setbox0=\hbox{$\mathchar"013C$}%
	\raise.2ex\copy0\kern-\wd0%
	\lower0.9ex\hbox{$\mathchar"0218$}}\mskip 3mu plus 2mu minus 1mu}}
\else
\def\simleq{\lesssim}
\fi

\ifx\gtrsim\undefined
\def\simgeq{{{\mskip 3mu plus 2mu minus 1mu%
	\setbox0=\hbox{$\mathchar"013E$}%
	\raise.2ex\copy0\kern-\wd0%
	\lower0.9ex\hbox{$\mathchar"0218$}}\mskip 3mu plus 2mu minus 1mu}}
\else
\def\simgeq{\gtrsim}
\fi

\newtheorem{theorem}{Theorem}
\newtheorem{lemma}[theorem]{Lemma}

\newtheorem{coro}[theorem]{Corollary}

\theoremstyle{definition}

\newtheorem{example}{Example}
\newtheorem{remark}{Remark}

\newif\ifmapx
{\catcode`/=0 \catcode`\\=12/gdef/mkillslash\#1{#1}}
\edef\jobnametmp{\expandafter\string\csname npmle2_apx\endcsname}
\edef\jobnameapx{\expandafter\mkillslash\jobnametmp}
\edef\jobnameexpand{\jobname}
\ifx\jobnameexpand\jobnameapx
\mapxtrue
\else
\mapxfalse
\fi

\long\def\apxonly#1{\ifmapx{\color{blue}#1}\fi}

\renewcommand{\hat}{\widehat}
\renewcommand{\tilde}{\widetilde}

\begin{document}
\ifpdf
\DeclareGraphicsExtensions{.pgf}
\graphicspath{{figures/}{plots/}}
\fi

\title{Self-regularizing Property of Nonparametric Maximum Likelihood Estimator in Mixture Models}

\author{Yury Polyanskiy and Yihong Wu\thanks{Y.P. is with the Department of EECS, MIT, Cambridge, MA, email: \url{yp@mit.edu}. Y.W. is with
the Department of Statistics and Data Science, Yale University, New Haven, CT, email: \url{yihong.wu@yale.edu}.}}

\maketitle

\begin{abstract}

Introduced by Kiefer and Wolfowitz \cite{KW56}, the nonparametric maximum likelihood estimator (NPMLE) is a widely used methodology for learning mixture models and empirical Bayes estimation. 
Sidestepping the non-convexity in mixture likelihood, the NPMLE estimates
the mixing distribution by maximizing the total likelihood over the space of probability measures, which can be viewed
as an extreme form of overparameterization.

In this paper we discover a surprising property of the NPMLE solution. Consider,
for example, a Gaussian mixture model on the real line with a subgaussian mixing distribution. Leveraging
complex-analytic techniques, we show that with high probability the NPMLE based on a sample of size $n$ has $O(\log n)$ atoms (mass points), significantly improving the deterministic upper bound of $n$ due to Lindsay \cite{lindsay1983geometry1}. Notably, any such Gaussian mixture is statistically 
 indistinguishable from a finite one with $O(\log n)$ components (and this is tight for certain mixtures).
Thus, absent any explicit form of model selection, NPMLE automatically chooses the right model complexity, a property we term \emph{self-regularization}. Extensions to other exponential families are given. As a statistical application, we show that this structural property can be harnessed to bootstrap existing Hellinger risk bound of the (parametric) MLE for finite Gaussian mixtures to the NPMLE for general Gaussian mixtures, recovering a result of Zhang \cite{zhang2009generalized}.
\end{abstract}

\tableofcontents

\section{Introduction}
\label{sec:intro}




\emph{Nonparametric maximum likelihood estimator} (NPMLE) is a useful methodology for various statistical problems such
as density estimation, regression, censoring model, deconvolution, and mixture models  (see the monographs
\cite{groeneboom1992information,groeneboom2014nonparametric}). 
Oftentimes optimizing over a massive (infinite-dimensional) parameter space can lead to undesirable properties, such as non-existence\footnote{For example, it is easy to see that NPMLE for the class of unimodal densities does not exist.} and roughness, and runs the risk of overfitting. These shortcomings can be remedied by the method of sieves \cite{grenander1981abstract} or explicit regularization \cite{good1971nonparametric,silverman1982estimation} at the expense of losing the main advantages of the NPMLE -- the full adaptivity (tuning parameters-free) and the computational tractability. 
However, for certain problems including shape constraints (such as monotonicity
\cite{grenander1956theory,birge1989grenader} and log-concavity
\cite{dumbgen2009maximum,cule2010maximum,doss2016global,kim2016global}) and mixture models
\cite{Lindsay1995,zhang2009generalized,saha2020nonparametric}, a striking observation is that \emph{unpenalized} NPMLE
achieves superior performance and has become the method of choice for both theoretical investigation and practical
computation.  While basic structural properties of NPMLE has been well understood, these results are frequently too
conservative to explain its superior statistical performance.
This paper studies the \emph{typical} structure of NPMLE for mixture models
as well as its statistical consequences.

Consider a parametric family of densities $\{p_\theta: \theta\in\Theta\}$ with respect to some dominating measure $\mu$ on $\reals$, where the parameter space $\Theta$ is assumed to be a subset of $\reals$. Given a mixing distribution (prior) $\pi$ on $\Theta$, we denote the induced mixture density as:
\begin{equation}
p_\pi(x) \triangleq \int_\Theta p_\theta(x) \pi(d\theta).
\label{eq:mixture}
\end{equation}
Introduced by Kiefer and Wolfowitz \cite{KW56} (see also an earlier abstract by Robbins \cite{robbins1950generalization}), the NPMLE for the mixing distribution is defined as a maximizer of the mixture likelihood given $n$ data points $x_1,\ldots,x_n$:
\begin{equation}
\NPMLEpi \in  \arg\max_{\pi \in \calM(\Theta)} \frac{1}{n} \sum_{i=1}^n \log p_\pi(x_i),
\label{eq:NPMLE}
\end{equation}
where $\calM(\Theta)$ denotes the collection of all probability measures on $\Theta$.
We refer the readers to the monograph of Lindsay \cite{Lindsay1995} for a systematic treatment on the NPMLE.
Although the convex optimization problem \prettyref{eq:NPMLE} is infinite-dimensional, over the years various computationally efficient algorithms have been obtained;
see \cite[Chapter 6]{Lindsay1995} and more recent developments in \cite{jiang2009general,koenker2014convex}.
The NPMLE provides a highly useful primitive for empirical Bayes and compound estimation problem, in which one first apply the NPMLE to learn a prior then execute the corresponding Bayes estimator of the learned prior. This strategy can be used as a universal means for denoising and achieves the state-of-the-art empirical Bayes performance \cite{jiang2009general}.


We summarize a few known structural properties of the NPMLE. The first existence and uniqueness result was obtained by Simar \cite{simar1976maximum} for the Poisson mixture, followed by Jewell \cite{Jewell1982} for mixtures of exponential distributions. 
It was shown that the (unique) solution $\NPMLEpi$ to the optimization problem \prettyref{eq:NPMLE} is a discrete distribution, whose number of atoms (mass points) is at most the number of distinct values of the observations and consequently at most the sample size $n$.\footnote{The existence of such an atomic maximizer is a direct consequence of Carath\'eodary theorem \cite[Chapter 2, Theorem 18]{eggleston1958convexity}; the uniqueness takes effort to show.}
These results have been significantly extended in
\cite{Laird1978,lindsay1983geometry1,lindsay1983geometry2,groeneboom1992information,lindsay1993uniqueness} which show
that the NPMLE solution is unique and $n$-atomic for all exponential families with densities with respect to the Lebesgue measure.
Although the bound $|\supp(\NPMLEpi)| \leq n$ is the best possible, 
which seems to suggest the estimator exhibits significant overfitting (since an $n$-component mixture requires $2n-1$ parameters to describe), 
in practice the support size is much smaller than $n$. Understanding this phenomenon is the main motivation behind this work.

To anchor the discussion, let us focus on the Gaussian location model, where $p_\theta(x)=\varphi(x-\theta)$ is the density of $N(\theta,1)$ and $\varphi(x)=\frac{1}{\sqrt{2\pi}} e^{-x^2/2}$ is the standard normal density, so that $p_\pi$ is the convolution $\pi * \varphi$. 
It is well known that for finite Gaussian mixtures, the likelihood is non-concave in the location parameters; furthermore, spurious local maxima can exist even with infinite sample size \cite{jin2016local} which pose difficulty for heuristic methods such as the EM algorithm. 
To sidestep the non-convexity, the approach of NPMLE can be viewed as an extreme form of \emph{overparameterization}, which postulates a potentially infinite Gaussian mixture so as to convexify the optimization problem. 
Since overparameterized models are prone to overfitting, it is of significant interest to understand the typical model size fitted by the NPMLE.
To this end, the worst-case bound $|\supp(\NPMLEpi)| \leq n $ is not useful. In fact, 
this bound can be tight, e.g., when the
$n$ observations are extremely spaced out \cite[p.~116]{Lindsay1995}. This, however, is \emph{not} a typical configuration of the sample if it consists of independent observations. Indeed, in practice it has been observed that 
NPMLE tends to fit a Gaussian mixture with much fewer components than $n$. 
This not only explains the absence of overfitting, but is a highly desirable property for interpretability  of the NPMLE solution, and is clearly not explained by the worst-case bound.
 Based on numerical evidence, Koenker and Gu \cite{koenker2019comment} suggested that the number of atoms of the NPMLE is typically $O(\sqrt{n})$. As our main result shows next, it is in fact $O(\log n)$.

\begin{theorem}[Gaussian mixture model]
\label{thm:gaussian}
	Let $p_\theta$ be the density of $N(\theta,1)$.
	Let $\xmin=\min_{i\in[n]} x_i$ and $\xmax=\max_{i\in[n]} x_i$. 
	Then there exists a universal constant $C_0$, such that 
\begin{equation}
|\supp(\NPMLEpi)|\leq C_0 (\xmax-\xmin)^2. 
\label{eq:npmle-gaussian}
\end{equation}	
	Consequently, suppose $x_1,\ldots,x_n$ are drawn independently from $\pi * N(0,1)$ for some $s$-subgaussian mixing distribution $\pi$, i.e., $\int \pi(d\theta) e^{t \theta} \leq e^{st^2/2}$ for all $t\in\reals$. 
	Then for any $\tau>0$, there exists some constant $C=C(s,\tau)$ such that with probability at least $1-n^{-\tau}$,
 	\begin{equation}
|\supp(\NPMLEpi)|\leq C \log n.
\label{eq:npmle-gaussian2}
\end{equation}	
\end{theorem}

A few remarks are in order:

\begin{remark}[Tightness of \prettyref{thm:gaussian}]
\label{rmk:gaussian-tight}	
	
	The $O(\log n)$ upper bound in \prettyref{thm:gaussian} is tight in the following sense:	
	\begin{itemize}
		\item 
		First, it is necessary to select a model of size $\Omega(\log n)$ in order to be compatible with existing statistical guarantees on the NPMLE. 
		Indeed, suppose the true density $p_\pi$ is $N(0,\sigma^2)$ for some $\sigma>1$ (i.e.~the mixing distribution is another Gaussian). It is known that the Hellinger distance between $N(0,\sigma^2)$ and any $k$-Gaussian mixture ($k$-GM) with unit variance is at least $\exp(-O(k))$  \cite{WV2010}. Therefore, if $|\supp(\NPMLEpi)| \leq c \log n$ for some small constant $c$, the bias would be too big, violating the 
Hellinger risk bound $\Expect[H^2(p_\pi,p_{\NPMLEpi})] = O(\frac{\log^2 n}{n})$ on the NPMLE \cite{zhang2009generalized} (see \prettyref{sec:stats}).
		\item 
		On the other hand, there is no statistical value to fit a model of size bigger than $\Omega(\log n)$. Indeed, it is easy to show by moment matching (see, e.g., 
		\cite[Lemma 8]{WY18})
		that
		for any subgaussian $\pi$, $p_\pi=\pi * N(0,1)$ can be approximated by a $k$-GM within total variation ($\TV$) distance $\exp(-\Omega(k))$. 
		Therefore, there exists a $k$-GM density $p_{\pi'}$ with $k=C \log n$, such that $\TV(p_{\pi},p_{\pi'}) = o(1/n)$. As such, one can couple the original sample 
		$X_1,\ldots,X_n$ drawn from $p_{\pi}$ with the sample $X'_1,\ldots,X'_n$ drawn from $p_{\pi'}$, so that with probability $1-o(1)$, 
		$X_i=X_i'$ for all $i=1,\ldots,n$. From this simulation perspective, $p_{\pi'}$ is an equally plausible ``ground truth'' that explains the data,
		and hence, statistically speaking, there is no reason to fit a mixture model with more than $C \log n$ components.
	\end{itemize}
		
	From the above two aspects, one can view $\Theta(\log n)$ as the ``effective dimension'' of the Gaussian mixture model 
with subgaussian mixing distributions (i.e. each doubling of the sample size \textit{unlocks} a new parameter of the
model class). Thus it is a remarkable fact that NPMLE picks up the right model size without explicit model selection penalty. 
	For this reason, we refer to the phenomenon described in \prettyref{thm:gaussian} as \emph{self-regularization}.
	In order to further quantify self-regularization and determine what the correct model size is, we formalize a framework called the statistical degree in \prettyref{sec:statdegree}.

\end{remark}

\begin{remark}[Poisson mixture]
\label{rmk:poisson}	
Using only classical results, 
one can get a glimpse of the self-regularization property of the NPMLE by considering the Poisson model. 
Suppose $x_1,\ldots,x_n$ are drawn independently from a Poisson mixture 
for some subexponential mixing distribution on the mean parameter. 
Since the observations are non-negative integers, the number of distinct values of in the sample is at most $\xmax+1$, which is $O(\log n)$ with probability $1-o(1)$ by a union bound. Thus the NPMLE for the Poisson mixture is $O(\log n)$-atomic with high probability, which is again the optimal model size.
Clearly, this argument does not generalize to continuous distributions such as the Gaussian mixture model in which all observations are distinct with probability one. 
Nevertheless, the \emph{range} of the data still grows logarithmically and \prettyref{thm:gaussian} shows that the number of atoms in the NPMLE can be bounded by the squared range.
\end{remark}

\begin{remark}[Model selection and penalized MLE]
\label{rmk:ms}	
Define the likelihood value of the best $k$-GM fit as
\begin{equation}
\Lopt(k) \triangleq  \max_{\pi\in\calM_k} \frac{1}{n} \sum_{i=1}^n \log p_\pi(x_i).
\label{eq:Lopt}
\end{equation}
Note that this is a non-convex optimization problem, since the likelihood is not concave in the location parameters. As $k\to\infty$, $\Lopt(k)$ approaches the objective value of the (convex optimization) NPMLE \prettyref{eq:NPMLE}. \prettyref{thm:gaussian} shows that with high probability with respect to the randomness of the sample, the curve $k \mapsto \Lopt(k)$ \emph{flattens} when $k$ surpasses $C \log n$ for some constant $C$.
This has the following immediate bearing on model selection. 
Various criteria (such as AIC or BIC \cite{leroux1992consistent,keribin2000consistent}) have been proposed for the mixture model: given a penalty function $\pen(k)$ that strictly increases in $k$, select a model size by solving
\[
\max_{k =1,\ldots,K} \sth{\Lopt(k)  - \pen(k)}
\]
where $K$ is a pre-defined maximal model size. \prettyref{thm:gaussian} shows that for Gaussian mixtures, regardless of the actual model size, there is no need to choose $K$ bigger than $C \log n$, which also suggests $K=C \log n$ a universal choice.
It is shown in \cite{leroux1992consistent,keribin2000consistent} that BIC (with $\pen(k)=\frac{k}{2} \log n$) is consistent in estimating the order of the mixture model.	
Complementing this result, \prettyref{thm:gaussian} shows that regardless of the choice of penalty, any penalized MLE will not choose a model size bigger than $C \log n$ with high probability.
\end{remark}

\begin{remark}[Comparison with shape-constrained estimation]
\label{rmk:shape}	
	
	The structure of the NPMLE is much less well understood for mixture models than for shape-constrained estimation. 
	For example, for monotone density, the NPMLE (known as the Grenander estimator \cite{grenander1956theory}) 
of a decreasing density on $[0,1]$ with $n$ observations is known to be piecewise constant with at most $n$ pieces. Denote by $k_n$ by its number of pieces. 
Under appropriate conditions it is shown that in general $k_n = O(n^{1/3})$ with high probability \cite[Lemma 3.1]{groeneboom2011vertices}. In the special case where the data are drawn from the uniform distribution, $k_n$ is asymptotically $N(\log n,\log n)$ \cite[Theorem 2]{groeneboom1993isotonic}. These results are made possible thanks to an explicit characterization of the NPMLE in terms of empirical processes, a luxury we do not have in mixture models.

On the other hand, there is a clear analogy for the structural behavior of NPMLE for monotone density and mixture
models: In the former, if the data are drawn from uniform distribution (one-piecewise constant), the NPMLE will fit a
piecewise constant density with $O(\log n)$ pieces; in the latter, if the data are drawn from a single Gaussian, the
NPMLE will fit a Gaussian mixture with $O(\log n)$ components. From this perspective one could say there is some mild
overfitting in NPMLE; nevertheless, it is a modest (and fair) price to pay for being completely automatic and computationally attractive. 
	

\end{remark}

\prettyref{thm:gaussian} is further extended in \prettyref{thm:crit} to general exponential families, which shows that there is some degree of universality to the $O(\log n)$ upper bound. As we will see in \prettyref{sec:kkt}, bounding the number of atoms in NPMLE boils down to counting critical points of functions of the form $F(\theta) = \sum_{i=1}^n w_i p_\theta(x_i)$, where $w_i$'s are nonnegative weights.
We accomplish this task using methods from complex analysis. 
Roughly speaking, the strategy is as follows:
First, we localize the roots of $F'$ in a compact interval, say $[-r,r]$. Then, 
we bound the number of zeros of $F'$ in the complex disk of radius $r$, 
 in terms of its maximal modulus on the complex disks. 
This leads to a deterministic upper bound, as a function of the sample, on the number of atoms of the NPMLE. 
Finally, we analyze the high-probability behavior of this upper bound when the sample consists of iid observations.
We note that in the special case of Gaussian model, counting the number of critical points of $F$ has been studied,
independently, in the context of a seemingly unrelated information-theoretic problem~\cite{dysto2020}; see also
\prettyref{sec:discuss-gmmaxima}.


Note that statistical guarantees on NPMLE, typically in terms of Hellinger risk of density estimation, have been obtained in \cite{ghosal.vdv,ghosal2007posterior,zhang2009generalized,saha2020nonparametric}.
These results follow the usual route of analyzing MLE using entropy numbers and only uses the zeroth order optimality
condition, and therefore cannot produce any \emph{structural} information on the optimizer such as the number of atoms.
(For example, such analysis applies equally to $\NPMLEpi$ convolved with an arbitrarily small Gaussian, which now has
infinitely many atoms.) A structural result, such as \prettyref{thm:gaussian}, can only be obtained by ``opening up the
optimization blackbox'', by examining the exact optimality conditions, as we indeed do below. In turn, a pleasant consequence of the self-regularizing property is a simpler proof of the statistical guarantee of the NPMLE in \cite{zhang2009generalized}, by bootstrapping existing that of the (parametric) MLE for finite Gaussian mixtures \cite{maugis2011non} to general mixtures.

\medskip
The remainder of the paper is organized as follows:
\prettyref{sec:kkt} recalls the first-order optimality condition for the NPMLE. 
Following \cite{lindsay1983geometry2}, \prettyref{sec:main} studies the NPMLE for mixtures of exponential family and bounds its number of atoms as well as analyzing its typical behavior. \prettyref{sec:stats} derives Hellinger risk bounds for the NPMLE in the Gaussian mixture model. 
\prettyref{sec:discuss} concludes the paper by discussing the concept of ``self-regularization'', its ramifications and open problems.

Throughout the paper, we use standard asymptotic notations: For any sequences $\{a_n\}$ and $\{b_n\}$ of positive numbers, we write $a_n \gtrsim b_n$ if $a_n\geq cb_n$ holds for all $n$ and some absolute constant $c > 0$, $a_n\lesssim b_n$ if $a_n \gtrsim b_n$, and $a_n \asymp b_n$ if both $a_n\gtrsim b_n$ and $a_n\lesssim b_n$ hold; the notations $O(\cdot)$, $\Omega(\cdot)$, and $\Theta(\cdot)$ are similarly defined. We write $a_n=o(b_n)$ or $b_n=\omega(b_n)$ or $a_n\ll b_n$ or $b_n \gg a_n$ if $a_n/b_n\to 0$ as $n\diverge$.

\section{Optimality condition}
\label{sec:kkt}
In this section we review the first-order optimality condition (both necessary and sufficient) for characterizing the NPMLE. 
	Denote the objective function in \prettyref{eq:NPMLE} by
	\[
	\ell(\pi) = \frac{1}{n} \sum_{i=1}^n \log p_\pi(x_i).
	\]
	Let $\hat\pi=\NPMLEpi$. 
	Since 
	$\ell(\hat\pi)\geq \ell((1-\epsilon)\hat\pi+\epsilon \delta_\theta)$ for any $\epsilon\in[0,1]$ and any $\theta\in \reals$, we arrive at the first-order optimality condition
	$\frac{d}{d\epsilon} \ell((1-\epsilon)\hat\pi+\epsilon \delta_\theta) \big|_{\epsilon=0} \leq 0$, namely,\footnote{The condition \prettyref{eq:kkt} is also sufficient for the global optimality of $\hat \pi$. Indeed, for any $\pi$, by the concavity of $\ell$, Jensen's inequality implies $\ell(\hat\pi) - \ell(\pi) \geq \frac{1}{\epsilon}[\ell(\hat\pi)-\ell((1-\epsilon)\hat\pi+\epsilon \pi)]$ all $\epsilon \in (0,1)$. Sending $\epsilon\to 0$ yields $\ell(\hat\pi) - \ell(\pi) \geq -\frac{d}{d\epsilon} \ell((1-\epsilon)\hat\pi+\epsilon \delta_\theta)\big|_{\epsilon=0} = 1 - \int \pi(d\theta) D_{\hat\pi}(\theta) \geq 0$.}	
	\begin{equation}
	D_{\hat\pi}(\theta) \triangleq \frac{1}{n} \sum_{i=1}^n \frac{p_\theta(x_i)}{p_{\hat\pi}(x_i)} \leq 1, \quad \forall \theta\in\reals.
	\label{eq:kkt}
	\end{equation}
	Furthermore, averaging the LHS of \prettyref{eq:kkt} over $\hat\pi$ and using the definition of the mixture density in \prettyref{eq:mixture}, we have
\[
\int \hat\pi(d\theta)  D_{\hat\pi}(\theta)
= \frac{1}{n} \sum_{i=1}^n \frac{ \int \hat\pi(d\theta)  p_\theta(x_i)}{p_{\hat\pi}(x_i)} = 1.
\]
We conclude that 
\begin{equation}
\supp(\hat\pi) \subset \{\text{Global maximizers of $D_{\hat\pi}$}\}.
\label{eq:supp}
\end{equation}
In particular, the number of atoms of $\hat\pi$ is at most the number of critical points of $D_{\hat\pi}$.

\begin{example}[Poisson mixture]
\label{ex:poisson}
As a concrete example, let us consider the Poisson model, where $p_\theta(x) = \frac{\theta^x}{x!} e^{-\theta}$ and $x \in \integers_+$. 
Thus
\[
\frac{d}{d\theta} D_{\hat\pi}(\theta) = e^{-\theta} \pth{\sum_{i=1}^n w_i (x_i \theta^{x_i-1}  - \theta^{x_i})},
\]
for some nonnegative weights $\{w_i\}$. Note that the quantity inside the parenthesis is a polynomial of $\theta$ of degree at most $\xmax$. Therefore, the number of critical points of $D_{\hat\pi}(\theta)$ and hence the number of atoms of $\NPMLEpi$ are at most $\xmax$. This result is first observed\footnote{The derivation here differs slightly with the original argument of \cite{simar1976maximum} which treats the system of $\{1,x,\ldots,x^k,e^x\}$.} in \cite{simar1976maximum}, which slightly improves the bound $\xmax+1$ in \prettyref{rmk:poisson}. 
For other models, the first-order condition typically does not reduce to a polynomial equation.
\end{example}

\section{Exponential families}
	\label{sec:main}

Following \cite{lindsay1983geometry1,lindsay1983geometry2}, we consider the following exponential family.
Let $p_0$ be a base density (with respect to some dominating measure $\mu$) on $\reals$, whose moment generating function (MGF) and 
cumulant generating function is defined as 
\begin{equation}
L(\theta) = \Expect_{X\sim p_0}[e^{\theta X}], \quad
\kappa(\theta) = \log L(\theta),
\label{eq:MGF}
\end{equation}
and is assumed to be finite for all $\theta\in(\utheta,\otheta)$, where $\utheta,\otheta \in [-\infty,\infty]$.
Define the following exponential family of densities with natural parameter $\theta$:
\[
p_\theta(x)  = 
\exp(\theta x - \kappa(\theta)) p_0(x).
\]
Notable examples include:
\begin{itemize}
	\item Gaussian location model $N(\theta,s)$: $p_0=N(0,1)$, $L(\theta)=e^{\frac{\theta^2}{2s}}$ and $\kappa(\theta)=\frac{\theta^2}{2s}$.
	
	\item Poisson model $\Poi(e^\theta)$: $p_0=\Poi(1)$, $L(\theta)=\exp\pth{e^\theta-1}$ and $\kappa(\theta)=e^\theta-1$.
\end{itemize}

%
%
%
%
	%

We need the following facts on the MGF:
\begin{lemma}
\label{lmm:MGF}	
\begin{enumerate}
	\item $L(\theta)>0$  for all $\theta\in\reals$.
	
	\item 
	$\kappa$ is strictly convex and hence
	\[
	\mu(\theta) \triangleq \kappa'(\theta) = \frac{L'(\theta)}{L(\theta)}
	\]
	is strictly increasing in $\theta$.
	Furthermore, if the distribution $p_0$ is fully supported on $\reals$, then $\mu(\pm\infty)=\pm\infty$.
	
	\item $L$  has an analytic extension on the strip $\{z\in \complex: \utheta<\Re(z) < \otheta\}$. Furthermore, 
	for each disk $D(z_0,r)$ contained in this strip, with $z_0=x_0+i y_0$,
	\begin{equation}
	\sup_{z\in D(z_0,r)} |L(z)| \leq \max\{L(x_0-r),L(x_0+r)\}
	\label{eq:MGF-sup}
	\end{equation}
	and
	\begin{equation}
	\sup_{z\in D(z_0,r)} |L'(z)| \leq \inf_{\epsilon>0} \frac{1}{\epsilon}\max\{L(x_0-r-\epsilon),L(x_0+r+\epsilon)\}.
	\label{eq:MGF-sup1}
	\end{equation}
	
\end{enumerate}
	
\end{lemma}

	%
%


Next we focus on continuous exponential families for which the NPMLE solution is known to be unique \cite{lindsay1993uniqueness}.
The following is a deterministic bound on the number of atoms of the NPMLE. 
\begin{theorem}
\label{thm:crit} Fix $\xmin\le\min_{i\in[n]} x_i$ and $\xmax\ge\max_{i\in[n]} x_i$. 
Define $\thetamin=\mu^{-1}(\xmin)$,  $\thetamax=\mu^{-1}(\xmax)$. 
Let $r = \frac{\thetamax-\thetamin}{2}$, $a = \frac{\xmax-\xmin}{2}$, and $x_0=\frac{\xmax+\xmin}{2}$.
Assume that $\xmin \leq \mu(0) \leq \xmax$.
For each $\delta>0$ such that 
$\delta< \frac{1}{5} \min\{\thetamin-\utheta,\otheta-\thetamax\}$,
	\[
	|\supp(\NPMLEpi)|\leq
	\frac{N_1}{\log{2r+2\delta\over 2r+\delta}} 
	\]
	where 
\begin{align*}
			N_1 &= 2 (a+|\mu(0)|+|x_0|)(|\theta|_{\max} + 2\delta)  + \kappa_{\max} + 
\log \frac{|x|_{\max} +\frac{1}{\delta}}{\tau}\\
		 \tau&=\max\{\mu(\thetamax+\delta) - \xmax, \xmin-\mu(\thetamin-\delta)\}\\
		 |\theta|_{\max}&=\max\{\thetamax,-\thetamin\}\\
		 |x|_{\max}&=\max\{\xmax,-\xmin\}\\
		 \kappa_{\max} &= \kappa(\thetamin-3\delta) \vee \kappa(\thetamax+3\delta)\,.
\end{align*}		 
\end{theorem}
\begin{remark} Roughly speaking, by choosing $\delta\asymp r$, \prettyref{thm:crit} shows that $ |\supp(\NPMLEpi)|\lesssim |\theta|_{\max} a + \kappa_{\max}$.
\end{remark}

\begin{proof}
Starting from \prettyref{eq:supp}, we bound the number of critical points of the following function
\begin{equation}
F(\theta) \triangleq \sum_{i=1}^n w_i \frac{p_\theta}{p_0}(x_i) = \sum_{i=1}^n w_i \exp(\theta x_i - \kappa(\theta)),
\label{eq:F}
\end{equation}
where $\sum_{i=1}^n w_i =1$ and $w_i = c{p_0(x_i)\over p_{\hat \pi}(x_i)}$ and $c$ is the normalization constant.
Then
\[
F'(\theta) =  \sum_{i=1}^n w_i \exp(\theta x_i - \kappa(\theta))[x_i - \mu(\theta)],
\]
Since $\mu=\kappa' = \frac{L'}{L}$ and $L(\theta)$ has no real roots (\prettyref{lmm:MGF}), we conclude that the critical points of $F(\theta)$ are the real roots of the following function:
\begin{equation}
G(\theta) \triangleq \sum_{i=1}^n w_i \exp(\theta x_i )[x_i L(\theta) - L'(\theta)].
\label{eq:G}
\end{equation}

We first notice that all real roots of $G$ should be on $[\thetamin,\thetamax]$. Indeed, by the strict monotonicity of $\mu$, we have $G(\theta)>0$ for $\theta>\thetamax$ and $G(\theta)<0$ for $\theta<\thetamin$.

Next, let us extend definition~\eqref{eq:G} to a complex argument $z$ and modify the function by introducing:
$$ 
g(z) = G(z+\theta_0) e^{-(z+\theta_0) x_0} = \Expect[e^{(z+\theta_0)(Y-x_0)}  (Y L(z+\theta_0) - L'(z+\theta_0))], \qquad z\in \mathbb{C}
$$
where $\theta_0=\frac{\thetamin+\thetamax}{2}$ and $x_0 = \frac{\xmin+\xmax}{2}$, and $\PP[Y=x_i] = w_i$. 
Note that the number of zeros of $g$
in $z \in [-r,r]$ is the same as the total number of real zeros of $G$. We will overbound this quantity by counting all
zeros of $g$ in a disk of radius $r$ on $\mathbb{C}$. To that end, we define $M_g(\rho) \eqdef \sup\{|g(z)|: |z| \le \rho\}$. 
We next fix $\delta_4>\delta_3>\delta_2>0$ such that $\thetamax + \delta_4 < \otheta$ and $\thetamin - \delta_4 >
\utheta$. Set $r_2 = r+\delta_2$ and $r_1=r + \delta_3$. 
On one hand, since $\mu(\thetamax)=\xmax$ and $\mu(\thetamin)=\xmin$, we have
\[
M_g(r_2) \geq |g(r_2)| = |G(\thetamax+\delta_2)| e^{-x_0(r_2+\theta_0)} \geq  e^{-a(|\theta|_{\max}+\delta_2)-x_0(\thetamax+\delta_2)} (\mu(\thetamax+\delta_2) - \xmax
) \cdot L(\thetamax+\delta_2) 
\]
and similarly
\[
M_g(r_2) \geq |g(-r_2)| = |G(\thetamin-\delta_2)| e^{-x_0(-r_2+\theta_0)} \geq  e^{-a(|\theta|_{\max}+\delta_2) -x_0(\thetamin-\delta_2)} (\xmin-\mu(\thetamin-\delta_2))
\cdot L(\thetamin-\delta_2) 
\]
By the convexity of $\kappa$ we have
$\kappa(\thetamax+\delta_2) \geq (\thetamax+\delta_2)\mu(0)$ and 
$\kappa(\thetamin-\delta_2) \geq (\thetamin-\delta_2)\mu(0)$.
Defining $\tau = \max(\mu(\thetamax+\delta_2) - \xmax, \xmin-\mu(\thetamin-\delta_2))$, we thus obtain
\begin{equation}
M_g(r_2) \geq e^{-(a+|\mu(0)|)(|\theta|_{\max}+\delta_2)-|x_0| |\theta|_{\max}} \tau.
\label{eq:Mg1}
\end{equation}

On the other hand, by \prettyref{lmm:MGF} we have
\[
\sup_{|z|\leq r_1}|L'(\theta_0+z)| \leq 
\frac{1}{\delta_4-\delta_3}\sup_{|z|\leq r_1+\delta_4}  |L(\theta_0+z)| 
= \frac{1}{\delta_4-\delta_3} L_4, \quad L_4 \eqdef L(\thetamin-\delta_4) \vee L(\thetamax+\delta_4)
\]
Since $\sup_{|z|\le r_1} |L(z)| \le  L(\thetamin-\delta_3) \vee L(\thetamax+\delta_3) \le L_4$ we conclude
\begin{equation}
M_g(r_1) \leq e^{a(|\theta|_{\max} + \delta_3)} \pth{|x|_{\max} +\frac{1}{\delta_4-\delta_3}} L_4
\label{eq:Mg2}
\end{equation}

Now setting $\delta_4 = 3\delta, \delta_3=2\delta, \delta_2 =\delta$ we get
$$ \log {M_g(r_1)\over M_g(r_2)} \le N_1\,.$$
The result then follows by the following lemma after also simplifying
$$ {r_1^2+r_2 r\over r_1(r_2 + r)} \ge {r_1 + r \over r_2 + r} = {2r + 2\delta\over 2r +\delta}\,. $$
\end{proof}

\begin{lemma}\label{lmm:countz} Let $f$ be a non-zero holomorphic function on a disk of radius $r_1$. Let $n_f(r) \eqdef |\{z: |z|\le r, f(z)=0\}|$ and
$M_f(r) \eqdef \sup_{|z|< r} |f(z)|$. For any $r<r_2<r_1$ we have
	$$ n_f(r) \le {1\over \log {r_1^2+r_2 r\over r_1(r_2 + r)}} \log {M_f(r_1)\over M_f(r_2)}\,.$$
	This bound is achieved by $f(z) = \left(r-z\over 1-rz\right)^n$.
\end{lemma}
\begin{proof} Without loss of generality we assume $r_1=1$. If $M_f(1)=\infty$ then there is nothing to prove. Otherwise, the bound is equivalent to showing
	\begin{equation}\label{eq:bla_x}
		M_f(r_2) \le M_f(1) C(r,r_2)^{-n_f(r)}\,, \qquad C(r,r_2) \eqdef {1+r_2 r \over r_2 + r} > 1
\end{equation}	
which means that every zero inside $rD$ reduces the magnitude of $f$ on the boundary of $r_2 D$ by a factor $C$. 
To show this, let us denote by $\{a_i\}$ the list of $n=n_f(r)$ zeros of $f$ inside $rD$ (with multiplicity). Thus we can write 
\begin{equation}\label{eq:bla_0}
	f(z) = g(z) \prod_{i=1}^{n} B_{a_i}(z)\,,
\end{equation}
where $B_a(z) \eqdef {|a|\over a} {a-z\over 1-\bar a z}$ is the Blaschke factor, and $g(z)$ is holomorphic on $D$ (and
has no zeros in the closed disk of radius $r$, but this is not going to be used below). Let us show that for any $|z|\le r_2$ and $|a|<r_2$ we have 
\begin{equation}\label{eq:bla_1}
	|B_a(z)| \le {|a|+r_2\over 1+|a|r_2}\,.
\end{equation}
Indeed, by the maximum principle it is sufficient to consider $z=r_2 e^{i\phi}, \phi \in [0,2\pi)$ and by rotating the disk, we can also assume
$a>0$. Then
\begin{equation}\label{eq:bla_2}
	|B_a(r_2 e^{i\phi})|^2 = {(a-r_2 \cos \phi)^2  + r_2^2 \sin^2 \phi\over 
			(1-ar_2 \cos \phi)^2 + a^2 r_2^2 \sin^2 \phi} = {a^2+r_2^2 - 2ar_2 \cos \phi \over 1+a^2 r_2^2 -
			2ar_2 \cos\phi}\,.
\end{equation}			
Since $a^2 + r_2^2 < 1+a^2 r_2^2$ we find that \prettyref{eq:bla_2} is maximized at $\phi = \pi$, thus proving~\eqref{eq:bla_1}. Furthermore,
from~\eqref{eq:bla_2} applied with $r_2=1$ we also note that $|B_a(z)|=1$ whenever
$|z|=1$, which via~\eqref{eq:bla_0} implies $M_g(1)=M_f(1)$.

Finally, from~\eqref{eq:bla_0}-\eqref{eq:bla_1} and the fact that $|g(z)|\le M_f(1)$ for all $z\in D$ we conclude that for any $|z|\le
r_2$ we have
	$$ |f(z)| \le M_f(1) \prod_{i=1}^n {|a_i|+r_2\over 1+|a_i|r_2}\,.$$
This concludes the proof of~\eqref{eq:bla_x} after noticing that each factor is upper bounded by ${1\over C(r,r_2)}$.
\end{proof}

As an application of \prettyref{thm:crit}, we now prove \prettyref{thm:gaussian} for Gaussian location mixtures.
\begin{proof}
Choose $\xmax=\max_{i\in[n]} x_i$ and $\xmin=\min_{i\in[n]} x_i$.
Recall that $p_\theta$ denote the density of $N(\theta,1)$. 
In this model, we have 
$\utheta=-\infty$, $\otheta=\infty$, 
$\kappa(\theta)=\theta^2/2$ and $\mu(\theta)=\theta$. 
Thus $\thetamin=\xmin$, $\thetamax=\xmax$, $r=a=\frac{1}{2}(\xmax-\xmin)$, $\tau=\delta$.
Conveniently, note that for \emph{location family}, we have the following translation invariance:
Let $T_x(\pi)$ denote the pushforward of $\pi$ under the translation $\cdot + x$. Then $\NPMLEpi(x_1+x,\ldots,x_n+x) = T_x(\NPMLEpi(x_1,\ldots,x_n))$.
Therefore, without loss of generality, we can assume $\xmin=-r \leq 0 \leq \xmax=r$, so that $x_0=0$ and $|x|_{\max}=r$.

	Choosing $\delta =r$ yields \prettyref{eq:npmle-gaussian}.
	Finally, the high-probability statement follows from $\prob{|x_i|\geq \tau} \leq \exp(-c \tau^2)$ for some constant $c$ and a union bound.
\end{proof}

	%
	%
	%

The examples of Gaussian and Poisson models (\prettyref{thm:gaussian} and \prettyref{ex:poisson}) seem to suggest that NPMLE is always $O(\log n)$-atomic with high probability. Indeed, there is some degree of universality to this bound, as the following result shows.  
The extra condition we impose essentially says that the tail probability $\Prob_{0}\{|X|\geq a\}$ behaves as $\exp(-a^c)$ for some $c>1$.
For notational convenience, we will assume that the base measure $p_0$ is symmetric around zero.
\begin{theorem}
\label{thm:logn} Fix $2<K_0\le K_1$ and $\theta_0,b,\beta>0$. Then there exist $n_0,C$ depending on $(K_0,K_1,\beta,b)$ with
the following property. Consider any density $p_0$ symmetric around zero whose log-MGF satisfies
\begin{equation}\label{eq:kappa_growth}
	K_0 \kappa(\theta) \le \kappa(2\theta) \le K_1 \kappa(\theta) \qquad \forall |\theta| >\theta_0\,.
\end{equation}
Let $x_1,\ldots,x_n \iiddistr p_\pi$ for some mixing distribution $\pi$ supported on the interval $[-b,b]$.
Then for all $n\geq n_0$, with probability $1-2n^{-\beta}$, $\NPMLEpi$  has at most $C\log n$ atoms.
\end{theorem}

\begin{remark}
	\prettyref{thm:logn} shows that the Gaussian tail is not essential for the $O(\log n)$ result to hold. 
	In fact, consider any smooth density $p_0$ such that $-\log p_0(x) \asymp |x|^\alpha$ for $\alpha>1$.	
	Then by saddle-point approximation we have 
	$\kappa(\theta) \asymp \theta^{\alpha/(\alpha-1)}$ as $\theta\to\infty$, which satisfies \prettyref{eq:kappa_growth}.
	
	On the other hand, compactly supported families are excluded since for those distributions $\kappa(\theta)$ is
	asymptotically linear (with a slope given by the essential supremum of $p_0$) as $\theta\to\infty$. Furthermore,
	exponential tails are also excluded. This is directly related to the open problem with mixtures of exponential distributions which will be discussed in \prettyref{sec:discuss-fail}.
\end{remark}

	%

\begin{proof}[Proof of \prettyref{thm:logn}]
We start by establishing properties of $\kappa(\cdot)$ and $\mu(\cdot)$ implied by conditions of the theorem.
	Under the symmetry assumption of $p_0$, $\kappa(\theta)$ is an even convex function with 
$\kappa(\theta) \geq \kappa(0)=0$ and $\mu(0)=0$. 
From the convexity of $\kappa$ we get for any $\theta>0$ 
	\begin{align*} \kappa(\theta) &\le \kappa(\theta/2) + {\theta \over 2} \mu(\theta) \\
			\kappa(2\theta) &\ge \kappa(\theta) + {\theta} \mu(\theta)
	\end{align*}
	And thus, for $\theta > \theta_0$ we get
		\begin{align} \mu(\theta) \theta &\ge C_0 \kappa(\theta), \quad C_0 = 2{K_0-1\over K_0} > 1
		\label{eq:gg1}\\
				\mu(\theta) \theta & \le C_1 \kappa(\theta), \quad C_1 = K_1-1 > 0\label{eq:gg2}
\end{align}		
Clearly, also, $\mu(\theta)\to\infty$ as $\theta\to\infty$ and hence $p_0$ is supported on the whole of $\mreals$. Thus we have $\utheta=-\infty$ and $\otheta=\infty$. 

Define the rate function for $a>0$: 
\[
E(a) \triangleq \sup_{\theta>0} a\theta-\kappa(\theta),
\]
which is achieved at $\theta=\rho\triangleq \mu^{-1}(a)$, so that $E(a) =  a\rho-\kappa(\rho)$.
From~\eqref{eq:gg1}-\eqref{eq:gg2} (noting $C_0>1$) we conclude that as $a\to\infty$ (and hence $\rho\to \infty$) we get:
\begin{equation}\label{eq:gg3}
	E(a) \asymp a \rho \asymp \kappa(\rho) 
\end{equation}

We need to establish one more consequence of~\eqref{eq:kappa_growth}. Namely, there exists $\theta_0' > 0$ and $m_0 \in
\naturals$
such that for any $\theta_1 > \theta_0'$ there exists $\theta^* \in [\theta_1, 2^{m_0-1} \theta_1]$ such that
\begin{equation}\label{eq:curvature}
	\mu(2\theta^*) - \mu(\theta^*) > 1\,.
\end{equation}
To show this, we select $m_0 > \log_2 {4 (K_0-1)\over K_0-2}$ and $\theta_0'\geq \theta_0$ so large that $\mu(\theta_0') \ge {4m_0\over K_0-2}$.
Now suppose (for the sake of contradiction) that for all $\theta \in [\theta_1, 2^{m_0-1} \theta_1]$ we have 
$$ \mu(2\theta) - \mu(\theta) \le 1\,.$$ 
Denoting $\theta_2 \eqdef 2^{m_0} \theta_1$, applying the above inequality repeatedly yields $\mu(\theta_2) \le \mu(\theta_1) + m_0$. Consequently, from the convexity
of $\kappa$ we have
$$ \kappa(\theta_2) \le \kappa(\theta_1) + (2^{m_0}-1) \theta_1 (\mu(\theta_1) + m_0)\,. $$
On the other hand, 
$$ \kappa(\theta_2/2) \ge \kappa(\theta_1) + (2^{m_0-1}-\theta_1) \mu(\theta_1)\,.$$
Taking the ratio of these, we get from~\eqref{eq:kappa_growth}:
$$ \kappa(\theta_1) + (2^{m_0}-1) \theta_1 (\mu(\theta_1) + m_0) \ge K_0 \left(\kappa(\theta_1) + (2^{m_0-1}-1)\theta_1
\mu(\theta_1)\right)\,.$$
Rearranging terms we arrive at 
$$ 2^{m_0} \theta_1 \left( \left({K_0\over 2} - 1\right) \mu(\theta_1) - m_0 \right) \le (\theta_1 \mu(\theta_1) - \kappa(\theta_1)) (K_0-1) -
\theta_1 m_0\,.$$
Dropping all negative terms on the right, and noticing that by the choice of $\theta_0'$ we have $({K_0\over 2} - 1)
\mu(\theta_1) - m_0 \ge {1\over 2} ({K_0\over 2} - 1) \mu(\theta_1)$, we conclude
$$ 2^{m_0} {1\over 2} \left({K_0\over 2} - 1\right) \theta_1 \mu(\theta_1)\le (K_0 - 1) \theta_1 \mu(\theta_1) \,.$$
By the choice of $m_0$, however, this is impossible. Hence, there must exist $\theta^*$ satisfying~\eqref{eq:curvature}.

Having established~\eqref{eq:gg3} and~\eqref{eq:curvature} we proceed to the proof of the theorem. Let $X \sim p_\pi$.
By the Chernoff bound, for any $\theta>0$, $\prob{X\geq a} \leq e^{-\theta a} \Expect[e^{\theta X}]$. Here 
\[
\Expect[e^{\theta X}] = 
 \int \pi(d\theta') \int dx p_0(x) e^{(\theta'+\theta) x-\kappa(\theta')} = 
 \int \pi(d\theta')\frac{L(\theta'+\theta)}{L(\theta')} \leq L(\theta+b),
\]
where the last inequality follows from the fact that $L$ is an even function lower bounded by $L(0)=1$, and 
$L(\theta'+\theta)\leq L(\theta+b)$ for any $\theta>0$ and $\theta'\in[-b,b]$.
Optimizing over $\theta$ we set $\theta=\rho-b$ and obtain
$
\prob{X \geq a} \leq e^{ab-E(a)},
$
provided that $\rho>b$.

Since we aim to apply Theorem~\ref{thm:crit}, we need to choose $\xmax$ and $\xmin$. We set them as follows. First we set
$\theta_1$ so that $a_1=\mu(\theta_1)$ verifies $E(a_1)-a_1b =  (1+\beta) \log
n$. Note that as $n\to\infty$ we have $a_1,\theta_1 \to \infty$. In the sequel, we assume that $n$ is so large that
$\theta_1> \theta_0'$ and $\theta_1 > b$. Notice that from~\eqref{eq:gg3} we have $E(a_1) \gg a_1$ and hence 
\begin{equation}\label{eq:gg4}
	E(a_1) \asymp \theta_1 a_1 \asymp \kappa(\theta_1) \asymp \log n\,. 
\end{equation}

Next, having selected $\theta_1$ we use~\eqref{eq:curvature} to select $\thetamax =
\theta^*$. We set $\xmax = \mu(\thetamax)$ and $\xmin = -\xmax, \thetamin =
-\thetamax$, so that $x_0=(\xmin+\xmax)/2=0$. 
 Then we have 
$\prob{X \geq \xmax} \leq \prob{X \geq a_1} \leq n^{-(1+\beta)}$.
Similarly, $\prob{X \leq -\xmin} \leq n^{-(1+\beta)}$. By the union bound this implies 
that with probability at least $1-2 n^{-\beta}$, we have $\xmin \le \min_i x_i \le \max_i x_i \le \xmax$.
Now we apply \prettyref{thm:crit} with $\delta=2\thetamax$, obtaining
	\begin{equation}
	|\supp(\NPMLEpi)|\lesssim 
		\thetamax \xmax  + \kappa(4\thetamax) +  \log \frac{\xmax + {1\over 2\thetamax}}{\tau},
	\label{eq:logn1}
	\end{equation}
		where $\tau=\mu(3\thetamax) - \mu(\thetamax) \ge \mu(2\thetamax) - \mu(\thetamax) > 1$
		by~\eqref{eq:curvature}. Consequently, the last term in~\eqref{eq:logn1} is dominated by the first.

Finally, we note that $\thetamax \in [\theta_1, 2^{m_0-1} \theta_1]$ and thus $\thetamax \asymp \theta_1$.
From~\eqref{eq:kappa_growth} we have $\kappa(\thetamax) \asymp \kappa(\theta_1)\asymp \log n$. Similarly,
from~\eqref{eq:gg3} we get $\thetamax \xmax \asymp \kappa(\thetamax) \asymp \log n$. In all, the right-hand side
of~\eqref{eq:logn1} is $\asymp \log n$ as claimed.
\end{proof}
	

\section{Statistical consequences on NPMLE}
	\label{sec:stats}
	
	In this section we show how the self-regularization property of the NPMLE allows one to ``bootstrap'' existing results on MLE in finite Gaussian models to infinite mixtures. 
	
The following statistical guarantee on NPMLE is due to Zhang \cite{zhang2009generalized}, improving over previous result of \cite{ghosal.vdv,ghosal2007posterior}.
\begin{theorem}
\label{thm:zhang}
	Let $X_1,\ldots,X_n \iiddistr p_\pi \triangleq \pi * \varphi$ and let $\hat\pi=\NPMLEpi(X_1,\ldots,X_n)$ be given in \prettyref{eq:NPMLE}. Then
	\begin{equation}
\sup_{\pi \in \calM_{\SG}(1)}	\Expect_\pi[H^2(p_{\hat \pi}, p_{\pi})] \lesssim \frac{\log^2 n}{n},
	\label{eq:zhang}
	\end{equation}
	where $\calM_{\SG}(s)$ denote the collection of all $s$-subgaussian distributions on $\reals$.
\end{theorem}
Next we show that using the self-regularization of the NPMLE, \prettyref{thm:zhang} can be deduced from existing guarantees on MLE in finite mixture models.
We need a couple of auxiliary results, whose proofs are deferred to the end of this section. The following result is on approximating a general Gaussian mixture by finite mixtures:
\begin{lemma}
\label{lmm:approx}	
	Let $\pi$ be 1-subgaussian. For any $a>0$ and any $k \in \naturals$, there exists a $k$-atomic $\pi'$ supported on $[-a,a]$, such that 
	\[
	\TV(p_\pi, p_{\pi'}) \leq 2 e^{-a^2/2} + 2 e^{a^2/4} \pth{\frac{e a^2}{2k}}^{k}
	\]
\end{lemma}
Next we recall the statistical guarantee on the parametric MLE in finite Gaussian mixtures. By standard results on MLE (cf.~e.g.~\cite{vandeGeer2000}), this can be deduced from the bracketing entropy for this class, which has been thoroughly investigated in the literature \cite{ghosal.vdv,genovese.wasserman,zhang2009generalized,maugis2011non}. The following result is a corollary of the entropy bound of Maugis and Michel in \cite{maugis2011non}.
\begin{lemma}
\label{lmm:pmle}	
Let $a\geq 1$ and $k \in \naturals$. 
Let $\hat \pi_{k,a} = \hat \pi_{k,a}(Y_1,\ldots,Y_n)$ is the (parametric) MLE defined in \prettyref{eq:parametric-MLE}, 
where $Y_i \iiddistr p_{\pi}$. There exists a universal constant $C$ such that
	\begin{equation}
\sup_{\pi\in\calM_{k,a}} \Expect_\pi[H^2(p_{\hat \pi_{k,a}}, p_{\pi})] \leq  \frac{C k}{n} \log \frac{na^2}{k}, \label{eq:pmle}
\end{equation}
where $\calM_{k,a}$ denotes the collection of all $k$-atomic distributions on $[-a,a]$.
\end{lemma}


\begin{proof}[Proof of \prettyref{thm:zhang}]
	Let $X_1,\ldots,X_n\iiddistr p_\pi = \pi * N(0,1)$ for some 1-subgaussian $\pi$. 
Define the event $E_0 \triangleq \{|X_{\max}| \leq a_0\}$, where $a_0= \sqrt{C_0\log n}$ for some large absolute constant $C_0$. 
Then $E_0$ has probability at least $1-n^{-2}$.
By \prettyref{thm:gaussian},
on the event $E_0$, 
$\hat\pi$ is supported on $[-a_0,a_0]$ and $|\supp(\hat \pi)| \leq C_1 a_0^2 = C_1C_0\log n \triangleq k_0$.
Then for any $k\geq k_0$ and $a\geq a_0$, on the event $E_0$, we have
\begin{equation}
\hat{\pi} = \hat{\pi}_{k,a}(X_1,\ldots,X_n) \triangleq \argmax_{\pi\in\calM_{k,a}} \sum_{i=1}^n \log p_\pi(X_i).
\label{eq:parametric-MLE}
\end{equation}
(In case that \prettyref{eq:parametric-MLE} has multiple maximizers, $\hat{\pi}$ is chosen to be any one of them.)

Pick $a=\sqrt{C_1 \log n}$ and $k=C_2 \log n$ such that $a \geq a_0$, $k \geq k_0$, and $a^2/k \leq 1/10$. Applying \prettyref{lmm:approx}	with this choice, we obtain a $k$-atomic distribution $\pi'$ supported on $[-a,a]$ such that $\TV(p_\pi, p_{\pi'}) \leq n^{-3}$. 
Let $Y_1,\ldots,Y_n\iiddistr p_\pi = \pi * N(0,1)$ for some 1-subgaussian $\pi$.
Then $\TV(\Law(X_1,\ldots,X_n), \Law(Y_1,\ldots,Y_n)) \leq n^{-2}$. Then there exists a coupling such that $X_i=Y_i$ for $i=1,\ldots,n$ with probability at least $1-n^{-2}$. Let $E_2$ denote this event.

On the event of $E_1 \cap E_2$, we have 
\[
\hat{\pi}=\hat{\pi}_{\NPMLE}(X_1,\ldots,X_n) = \hat{\pi}_{k,a}(X_1,\ldots,X_n) = \hat{\pi}_{k,a}(Y_1,\ldots,Y_n).
\]
Now we are in a position to pass the statistical guarantee on the parametric MLE in finite Gaussian mixtures to the NPMLE. Applying \prettyref{lmm:pmle} with $k\asymp \log n$ and $a \asymp \sqrt{\log n}$, we have 
	\begin{equation}
\Expect[H^2(p_{\hat \pi_{k,a}}, p_{\pi'})] \lesssim \frac{\log^2 n}{n}.
\end{equation}
Finally, using the fact that $H^2 \leq \TV$ and the triangle inequality for Hellinger, we have
\[
\Expect[H^2(p_{\hat \pi}, p_{\pi}) \indc{E_0\cap E_1}]
\leq 
2 \Expect[H^2(p_{\hat \pi_{k,a}(Y_1,\ldots,Y_n)}, p_{\pi'})]
+ 2 H^2(p_{\pi}, p_{\pi'}) \leq C_3 \frac{\log^2 n}{n}.
\]
The proof is completed since $H^2 \leq 2$ and $E_0\cap E_1$ has probability at least $1-2n^{-2}$.
\end{proof}

\begin{remark}
\label{rmk:de}	
The following minimax lower bound is shown in \cite{kim2014minimax}:
	\begin{equation}
\inf_{\hat p}\sup_{\pi \in \calM_{\SG}(1)}	\Expect_\pi[H^2(\hat p, p_{\pi})] \gtrsim \frac{\log n}{n},
	\label{eq:kim}
	\end{equation}
which differs from the upper bound in \prettyref{thm:zhang} by $\log n$. As frequently observed in the density estimation literature, such a logarithmic factor can be attributed to the fact that the analysis of the MLE is based on the global entropy bound. Thus obtaining a local version of the entropy bound in \cite{maugis2011non} can potentially close this gap and establish the sharp optimality of the NPMLE in achieving the lower bound in \prettyref{eq:kim}.	
\end{remark}


\begin{proof}[Proof of \prettyref{lmm:approx}]
	Without loss of generality, assume that $\pi$ has zero mean.
	Let $\tilde\pi$ denote the conditional version of $\pi$ on $[-a,a]$. By the data processing inequality of total variation,
	\[
	\TV(\pi*N(0,1), \tilde\pi*N(0,1)) \leq \TV(\pi, \tilde\pi) = \pi([-a,a]^c) \leq 2 e^{-a^2/2}
	\]
	where the last inequality follows from $\pi$ being 1-subgaussian.
	Next, let $\pi'$ denote the $k$-point Gauss quadrature of $\tilde\pi$, such that $\pi'$ and $\tilde \pi$ have identical first $2k-1$ moments, and $\pi'$ is also supported on $[-a,a]$. Then by moment-matching approximation (see 	\cite[Lemma 8]{WY18}), 
we have
        \[
            \chi^2(\pi'*N(0,1)\|\tilde\pi*N(0,1))\le 4 e^{a^2/2} \pth{\frac{ea^2}{2k}}^{2k}.
        \]
    %
		Using the fact that $2 \TV^2 \leq \chi^2$ and the triangle inequality, the previous two displays yield the desired bound.
\end{proof}

\begin{proof}[Proof of \prettyref{lmm:pmle}]
Let $N_{[]}(\epsilon)$ denote the bracketing number of the class of $k$-GM densities $\calP_{k,a}\triangleq \{p_{\pi}: \pi\in\calM_{k,a}\}$ with respect to the Hellinger distance.
Applying Eq.~(B.8) in \cite[Proposition B.4]{maugis2011non} (with $\alpha=Q=1$, $D(k,\alpha)=3k$, $\lambda_m=\lambda_M=1$, so that $\calI \asymp K \log a$), we have
\begin{equation}
\log N_{[]}(\epsilon) \lesssim k \log  \frac{a}{\epsilon},
 \label{eq:maugis}
\end{equation}
Next we can apply standard results on the density estimation guarantee (in Hellinger distance) for the MLE (see e.g.~\cite[Theorem $7.4$]{vandeGeer2000}). 
Define  $J(\epsilon) \triangleq \int_{\epsilon^2}^\epsilon \sqrt{\log N_{[]}(u)} du$. 
By \prettyref{eq:maugis}, we have $J(\epsilon) \lesssim \epsilon \sqrt{k \log \frac{a}{\epsilon}}$. 
Thus $\Expect[H^2(p_{\hat \pi_{k,a}}, p_{\pi})]  \lesssim \epsilon_n^2$, where $\sqrt{n} \epsilon_n^2 = J(\epsilon_n)$ so that $\epsilon_n \asymp \sqrt{\frac{k}{n} \log \frac{na^2}{k}}$.
\end{proof}

%
%
	%
	%
	%
	%
	%

\section{Discussions}
	\label{sec:discuss}
	

\subsection{Statistical degree}
\label{sec:statdegree}


In this subsection we discuss the concept of self-regularization.
Loosely speaking, an unregularized estimator can be said to achieve some form of self-regularization if it returns a density with $o(n)$ components, which improves over the worst-case upper bound of $n$.
Expanding on the reasoning in \prettyref{rmk:gaussian-tight}, below we introduce a formal framework and provide a perspective on what may be the correct model size.

%
%
%

Consider a sequence of nested statistical models $M_1\subset M_2 \subset \cdots M \subset \mathcal{M}(\mathcal{X})$, where $k$ is a parameter that encodes the ``model complexity'' of $M_k$. 
For example, in linear models, $M_k$ denotes those with $k$-sparse regression coefficients; in shape-constrained setting, $M_k$ can be the set of $k$-piecewise constant or log-affine densities; in our setting of mixture models, $M_k$ is the set of all $k$-GM densities.

Given a sample of size $n$, we define the statistical degree $K_n$ as
\begin{equation}\label{eq:def_kn}
	K_n \triangleq \inf\sth{k: d_{\max}(M, M_k) \le {1\over 3\sqrt{n}}}\,,
\end{equation}
where $d_{\max}(A,B) \triangleq \sup_{P\in A} \inf_{Q \in B} H(P,Q)$ denotes the best approximation error (in the Hellinger distance) of the model class $A$ by members of $B$. 
By definition, $K_n$ is the largest $k$ so that any density in $M$ can be made statistically indistinguishable (on the basis of $n$ observations) from some density in $M_k$; in other words, given a sample of size $n$ drawn independently from any $f \in M$, one can simulate it with probability at least $1-c$ for some constant $c$ using one drawn from some $f_k \in M_k$.
 From this simulation perspective, there is no statistical reason to fit a model of complexity bigger than $K_n$; on the other hand, it does not compromise the statistical performance (in terms of the Hellinger rate) to restrict to models of complexity at most $K_n$.
Thus, we view achieving the statistical degree $K_n$ as a criterion of self-regularization.
As shown in \prettyref{rmk:gaussian-tight}, for the class $M$ of Gaussian mixtures with subgaussian mixing distributions, we have $K_n =\Theta(\log n)$, which coincides with the typical model size fitted by the NPMLE.


Next, we discuss a simple example where the self-regularization of the unpenalized NPMLE can be established directly.
\begin{example}
\label{ex:simple}
Consider observations taking non-negative integer values in $\calX=\integers_+$.
For each $k\geq 1$, let $M_k$ denote the set of distributions supported on $\{0,\ldots,k\}$, 
and let $M$ the class of 1-subgaussian distributions on $\integers_+$. 
It is clear that the statistical degree in this case is $K_n = \Theta(\sqrt{\log n})$. Indeed, the upper bound follows from truncation and the uniform subgaussian tail, and the lower bound follows from considering an explicit distribution such as $P(j) \propto e^{-j^2}$.

Given $x_1,\ldots,x_n \iiddistr P \in M$, the NPMLE for $P$ (without enforcing the subgaussianity) is simply the empirical distribution $\hat{P}$, where 
\begin{equation}
\hat{P}(j) = \frac{1}{n} \sum_{i =1}^n \indc{x_i = j}, \quad j \in \integers_+.
\label{eq:npmle-simple}
\end{equation}
By a union bound, there exists a constant $C$, such that with probability $1-o(1)$, $\hat{P}(j)=0$ for all $j \geq k= C \sqrt{\log n}$. In other words, with high probability we automatically have $\hat{P} \in M_k$ for some $k$ that agrees with the statistical degree.
\end{example}

Note that the self-regularizing property in \prettyref{ex:simple} is a simple consequence of the explicit expression of the NPMLE in \prettyref{eq:npmle-simple}. 
In contrast, for mixture models in Theorems \ref{thm:gaussian} and \ref{thm:crit} we need to resort to the optimality condition and complex-analytic techniques, due to the lack of close-form expression of NPMLE in mixture models. Another major difference is that for mixture models $M_k$ is non-convex and hence optimizing the likelihood over $M_k$ can be expensive. 
Quite spectacularly, the full relaxation over all measures somehow automatically solves the nonconvex optimization (and for the right $k$). 

\apxonly{

Finally, we discuss the connection between statistical degree and metric entropy. 
As mentioned in \prettyref{rmk:gaussian-tight}, an upper bound on the statistical degree entails a positive result on approximation which can be shown using techniques such as moment matching. For the lower bound, in addition to proving an inapproximability result  for a specific choice of distribution in $M$, an alternative way is via metric entropy. To this end, let us define the auxiliary quantity called the entropy degree $K'_n$ as follows:
$$ K'_n \triangleq \sup\{k: N_{cov}(M, 2/\sqrt{n}) > N_{cov}(M_k, 1/\sqrt{n})\}\,,$$
where $N_{cov}(A, \epsilon) = \inf\{|S|: S\subset \mathcal{M}(\mathcal{X}), d_{\max}(A,S) \le \epsilon\}$ denotes the $\epsilon$-covering number of $A$ in Hellinger. So $K'_n$ is the $k$ for which $M_k$ has roughly the same entropy numbers (at scale $1/\sqrt{n}$) as the full model $M$.

The next result shows a one-sided bound but we expect $K_n \asymp K'_n$ in most natural models.
\begin{lemma} We always have $K_n \ge K'_n$.
\end{lemma}
\begin{proof}
	Consider any set $S$ achieving the infimum in the definition of $N_{cov}(M_k,1/\sqrt{n})$ for $k=K'_n$. 
	By the triangle inequality this set $S$ is also a ${2\over \sqrt{n}}$-covering of the set $V=\{P \in M:
	\inf_{Q\in M_k}H(P, Q) \le 1/\sqrt{n}\}$. However, by the definition of $K'_n$ the set $S$ is too small to
	provide a ${2\over \sqrt{n}}$-covering of $M$, and thus there must exist some $P \in M \setminus V$. In turn,
	this implies $d_{\max}(M,M_k) \ge {1\over \sqrt{n}}$.
\end{proof}

}

\apxonly{
\paragraph{When does self-regularization fail?}
Surprisingly, we were not able to find a natural example of a sequence of models $M_k \subset M$ with the property that NPMLE
selects a distribution $\NPMLEpi \in M_{\hat k}$ with $\hat k \gg K_n$. However, for the example of $\hat k \ll K_n$ we
can consider the following. Let $M$ be the set of all bounded monotone densities on $[0,1]$ and $M_k$ be the subset of
$k$-piecewise-constant densities. In this case, it is known that $\hat k \lesssim n^{1/3}$, whereas considering $\pi(x)
= x^2$ we can show $K_n \gtrsim n^{1/2}$. (However, notice that the minimax rate of density estimation in Hellinger over
the class $M$ should be $\asymp n^{-1/3}$, which means this mismatch would disappear had we defined $K_n$ using
the minimax rate of $n^{-1/3}$ \cite{birge1989grenader} in~\eqref{eq:def_kn} instead of the $n^{-1/2}$ that required statistical
indistinguishability.)
}

\subsection{Self-regularization for mixtures of exponentials}
\label{sec:discuss-fail}


Although we have not identified an example of a mixture model where the number of atoms of NPMLE is $\omega(\log n)$, the program of analyzing the NPMLE in \prettyref{thm:crit} and \prettyref{thm:logn} does have its limitations.
As a leading example, let us consider mixtures of exponential distributions, which is among the earliest results on the structure of NPMLE \cite{Jewell1982} (see also \cite[Sec.~2.1]{groeneboom1992information}). Since the tail is exponential, this model is outside the scope of \prettyref{thm:logn}.

\begin{example}[Exponential mixture]
\label{ex:exp-mixture}
Consider the exponential distribution $\Exp(\theta)$ with density $p_\theta(x)=\theta e^{-\theta x} \indc{x>0}$ and $\theta > 0$. 
In this case, the NPMLE is defined as 
\begin{equation}
\NPMLEpi =  \arg\max_{\pi \in \calM(\reals_+)} \frac{1}{n} \sum_{i=1}^n \log p_\pi(x_i), \quad p_\pi(x) = \int \theta e^{-\theta x} \pi(d\theta).
\label{eq:NPMLE-exp}
\end{equation}
Upon normalization, the gradient \prettyref{eq:kkt} is proportional to the function
\begin{equation}
F(\theta) = \sum_{i=1}^n w_i \theta e^{-\theta x_i},
\label{eq:F-expmix}
\end{equation}
 where $\sum_{i=1}^n w_i=1$ and $w_i\geq 0$.
Thus the atoms of the NPMLE are roots of 
$F'(\theta) = \sum_{i=1}^n w_i e^{-\theta x_i}(1-\theta x_i)$, which are localized in the interval $[a,b]$ with $a=1/\xmax$ and $b=1/\xmin$.
Following the proof of \prettyref{thm:crit}, to bound the number of roots of $F'$, we can apply \prettyref{lmm:countz} to $f(\theta) = F'(\theta - \frac{a+b}{2})$ and $r=\frac{b-a}{2}$. Choose $r_2=\frac{a+b}{2}$ and $r_1=2r=b-a$. Since $f(r_2)=F'(0)=1$, we have $M_f(r_2) \geq 1$. Moreover, it is clear that $M_f(r_1) \leq \exp(C b \xmax)$ for some constant $C$. Thus an application of \prettyref{lmm:countz} shows that 
\begin{equation}
|\supp(\NPMLEpi)| \lesssim \frac{\xmax}{\xmin}. 
\label{eq:npmle-exp}
\end{equation}

However, in the stochastic setting the above bound is too loose to be useful. Indeed, suppose $x_1,\ldots,x_n$ are drawn independently from a single exponential distribution, say, $\Exp(1)$. Then with high probability, we have $\xmin = \Theta_P(\frac{1}{n})$ and 
$\xmax = \Theta_P(\log n)$.
Thus \prettyref{eq:npmle-exp} yields $|\NPMLEpi| = O(n \log n)$, which is even worse than the deterministic bound of $|\NPMLEpi|\leq n$.
Clearly, the culprit of this looseness stems from the fact that the data-generating distribution is supported on $\reals_+$ which has a boundary at zero.
Since the smallest observation will be on the order of $\frac{1}{n}$, \emph{a priori} one can only localize the atoms of the NPMLE in an interval of width $\Theta(n)$, which is much worse than $\Theta(\sqrt{\log n})$ in the Gaussian model.
Similar problems also arise in other distributions whose support has boundary points, such as Gamma or Beta families.

\textit{Open question:} Given $X_i \iiddistr p_\pi$ where $\supp(\pi)\subset[1,2]$, prove that with probability $1-o(1)$, we have
	\begin{equation}\label{eq:conj_expon}
		|\supp(\NPMLEpi)| = \Theta(\log n)  
\end{equation}
The crucial $O(\log n)$ upper bound would follow from the following analytic \textit{conjecture:} For any distribution $\pi$ on $[-a,a]$ the convolution $(\pi * h)(x) \eqdef \int h(x-y) \pi(dy)$ has at most $O(a)$ critical points, where $h(x) = e^{-e^{x}+x}$ is the density of a Gompertz distribution.
\end{example}


\subsection{Compactly supported NPMLE}
\label{sec:discuss-compact}

So far we have focused on unconstrained NPMLE, where the likelihood is maximized over all mixing distributions. In case where one has extra knowledge such as compact support, moment constraint, or sparsity, these information can be incorporated into the optimization problem as linear constraints leading to potentially improved statistical performance.
This begs the question: to what extent does constraint help the self-regularization of the NPMLE. Specifically,
\begin{enumerate}
	\item If the unconstrained solution fails to self-regularize, does adding constraints make it so? 
	\item If the unconstrained solution is already self-regularizing, does adding constraints make it more so?
\end{enumerate}
We briefly discuss these two aspects below.

For the first problem, let us continue \prettyref{ex:exp-mixture} on exponential mixtures, where we pointed out that the
program in \prettyref{thm:crit} does not resolve the self-regularization of unconstrained NPMLE. Nevertheless, it is
easy to show that adding a support constraint to NPMLE does resolve conjecture~\eqref{eq:conj_expon}. Indeed, suppose that the parameter $\theta$ is bounded from above by some constant $\theta_0$, in which case one can consider the following support-constrained version of \prettyref{eq:NPMLE-exp}:
\begin{equation}
\NPMLEpi' =  \arg\max_{\pi \in \calM([0,\theta_0])} \frac{1}{n} \sum_{i=1}^n \log p_\pi(x_i).
\label{eq:NPMLE-exp1}
\end{equation}
Thanks to the constraint, we only need to count the number of critical points of \prettyref{eq:F-expmix} in the interval
$[0,\theta_0]$. Applying the same argument in \prettyref{ex:exp-mixture} now with $r=\theta_0/2=O(1)$ yields
$|\NPMLEpi'| \lesssim \xmax = O_P(\log n)$. Note that using moment matching and Taylor expansion we can show that the
statistical degree for exponential mixtures with parameters bounded away from zero and infinity (say, $\supp(\pi)\subset[1,2]$) is $O(\log n)$. We \textit{conjecture} that the
statistical degree $K_n$ in this case is $\Theta(\log n)$ and if so, the argument above shows that $\NPMLEpi'$ does
self-regularize.

For the second problem, let us revisit the Gaussian location mixture. Suppose the mixing distribution is supported on a compact interval, say, $[-1,1]$. \prettyref{thm:gaussian} shows that the unconstrained NPMLE is $O(\log n)$-atomic with high probability. However, when the mixing distribution is compactly supported, the moment-matching argument in 
\cite[Lemma 8]{WY18} shows that the statistical degree in fact reduces to $O(\frac{\log n}{\log\log n})$. Again, we
\textit{conjecture} that in this case $K_n \asymp {\log n \over \log \log n}$. Then a natural question is whether NPMLE
with support constraint $\NPMLEpi'$ defined as in~\eqref{eq:NPMLE-exp1} with maximization over $\{\pi: \supp(\pi) \in
[-1,1]\}$ achieves a better self-regularization of $O(\frac{\log n}{\log\log n})$ atoms.\footnote{It would
be even more spectacular if the unconstrained NPMLE achieved the same number of atoms, but we are not willing to
conjecture this.}

The main bottleneck of proving this is the following. Note that similar to the proof of
Theorem~\ref{thm:crit} we can reduce to the problem of counting the critical point of \prettyref{eq:F}, which for Gaussian model simplifies to
\begin{equation}
F(\theta) = \sum_{i=1}^n w_i \varphi(\theta-x_i), \quad w_i \propto \frac{1}{(\NPMLEpi' * \varphi)(x_i)}.
\label{eq:F-gaussian}
\end{equation}
However this time we are not interested in bounding the number of all critical points of $F$, but only those in $[-1,1]$. 
Thus, we can set $r=1$, $r_1 \asymp \sqrt{\log n}$ in the application of Lemma~\ref{lmm:countz}. The issue is with
setting $r_2$. If we could show that $F$ must have at least one point $z_0$ inside a disk of radius $O(1)$ such that
\begin{equation}\label{eq:Fconj}
	|F'(z_0)| > n^{-C}  
\end{equation}for some $C$ (with high probability), then invoking \prettyref{lmm:countz} with $r_2 = O(1)$ would conclude
that $F'$ has at most $O(\frac{\log n}{\log\log n})$ roots inside the unit disk. It is tempting to conjecture further that
$z_0$ satisfying~\eqref{eq:Fconj} exists for arbitrary choice of $\{w_i,x_i\}_{i=1}^n$, s.t. $|x_i| \lesssim
\sqrt{\log n}$.
Alas, this stronger conjecture does not hold as~\cite[Section 2]{PW20inapprox} constructs $\{w_i,x_i\}_{i=1}^{O(\log n)}$ such that
$|F'(z)|\leq n^{-C \log \log n}$ for all $|z| = O(1)$. 
Therefore unlike the proof of \prettyref{thm:gaussian}, here we cannot ignore the stochastic origin of $x_i$ and that
$w_i$ is inversely proportional to the fitted likelihood at $x_i$ (see \prettyref{eq:F-gaussian}). Since $\NPMLEpi'$ itself is random, proving this property of $G(z)$ seems to require a delicate analysis of 
``small-ball'' probabilities of the empirical process. This is left for future work.

\subsection{Maxima of Gaussian mixtures}\label{sec:discuss-gmmaxima}

In the special case of the Gaussian location mixture,
\prettyref{thm:crit} translates to the following statement on the Gaussian convolution:
For any distribution $\pi$ supported on the interval $[-a,a]$, the convolution $\pi*\varphi$ has at most $O(a^2)$ critical points.
This result has been shown independently in the recent work \cite[Theorem 6]{dysto2020} by similar techniques using a
corollary of Jensen's formula from \cite{tijdeman71}. Can this bound be improved? The answer is negative and we next
give a simple construction of a Gaussian mixture with $\Omega(a^2)$ local maxima.\footnote{A different construction
using $\Omega(a^2)$ equally weighted and equally spaced Gaussians is given in the independent work \cite{KK20} in
response to a conjecture of authors of~\cite{dysto2020}, cf. \textit{arxiv:1901.03264v4}, that their bound can be
improved to $O(A)$.}

\apxonly{The optimality of the $O(a^2)$ bound can also be deduced from existing
information-theoretic results. Indeed, in the problem of maximizing capacity of the additive white-Gaussian noise
channel over inputs $X \in [-a;a]$ (the amplitude constraint), it is known that the optimal distribution $P_X^*$ is discrete with atoms
located at the global maxima of a certain Gaussian convolution (TODO: this is correct, but mixing distribution is
non-compactly supported. Otherwise we can relate number of atoms to number of level-crossings of the mixture-PDF). Then combining the amplitude-constrained capacity lower
bound \cite{PW13} and the cardinality-constrained capacity upper bound \cite{WV2010} for the Gaussian channel, we
conclude that the number of atoms of $P_X^*$ is $\Omega(a^2)$.}

%
%

\begin{lemma}
\label{lmm:sinusoid} Let $h$ be a continuous probability density on $\reals$ with characteristic function $\hat h$ and CDF $H$. Suppose we have $\omega_0$
and $a>0$ such that
	\begin{equation}
	|\hat h(\omega_0)| > 2(H(-a/2) + 1-H(a/2))\,.
	\label{eq:sinusoid-ass}
	\end{equation}
	Let $\pi(x) = c (1+\sin(\omega_0 x)) 1\{|x| \le a\}$ with $c>0$ chosen to make $\pi$ a probability density. Then $h*\pi$
has at least $\omega_0 a\over 2\pi$ local maxima on $[-a/2,a/2]$.
\end{lemma}

\begin{proof} Let $\pi_0(x) = 1+\sin(\omega_0 x)$. Then $0 \leq \pi_0  \leq 0$. Then $(\pi_0 * h)(x) = |\hat h(\omega_0)| \sin(\omega_0 x -
\arg \hat h(\omega_0)) + 1$, which is a shifted and scaled sinusoid.
 Let $S_+$ and $S_-$ be the sets of global maxima and minima of $\pi_0*h$ (which are lattices with
step $2\pi \over \omega_0$). Let $\pi_1(x) = \pi_0(x) 1\{|x| \le a\}$. Define 
$\Delta \triangleq h*\pi_0 - h*\pi_1$. Then $\Delta \geq 0$ everywhere. Furthermore, 
\begin{equation}
\Delta(x)  = \int_{|y|>a} \pi_0(y) dh(x-y) \le 2(H(x-a) + 1-H(x+a))\,.
\label{eq:truncation}
\end{equation}
Thus, by assumption \prettyref{eq:sinusoid-ass}, for any $|x| \le a/2$ we have $\Delta(x) \le |\hat h(\omega_0)|$. Consequently, for any $x\in S_+ \cap [-a/2,a/2]$
we have $(\pi_1*h)(x) = (\pi_0*h)(x) - \Delta(x) > 1$ and for any $x \in S_- \cap [-a/2, a/2]$ we have $(\pi_1*h)(x) < 1-|\hat
h(\omega_0)|$. Thus, the level $1-{1\over2}|\hat h(	_0)|$ must be crossed in between any two consecutive points
from $S_+$ and $S_-$, implying the statement.
\end{proof}

\begin{coro}
\label{cor:gaussian-modes}
 There exists a compactly supported density $\pi$ on $[-a,a]$ so that $\pi*\varphi$ has $\Omega(a^2)$
local maxima on $[-a/2,a/2]$.
\end{coro}
\begin{proof}
	By the Gaussian tail bound, we have $H(-a/2) + 1-H(a/2) \leq 2 e^{-a^2/8}$. Choosing $\omega_0=a/4$, the claim follows from \prettyref{lmm:sinusoid} for sufficiently large $a$.
\end{proof}


\subsection{Mixture of log-concave densities}
\label{sec:discuss-logconcave}

Consider the following question: Given a convex combination of $k$ unimodal densities, how many modes can it have? A moment of thought shows that 
the answer is trivial as the sum of two unimodal densities,
e.g.~$f(x)+f(x-1)$ with 
\begin{equation}
f(x)=(1-|x|) \indc{|x|\leq 1},
\label{eq:unimodal}
\end{equation}
 can have infinitely many modes; the same example also applies even if unimodality is replaced by log-concavity. 
A natural question is what happens to strongly log-concave densities.\footnote{Recall that (cf.~\cite[Definition 2.9]{saumard2014log}) a density $f$ is called $c$-strongly log-concave strongly convex if $\log f$ is strongly concave, i.e., $\log f((1-\alpha)x + \alpha y) \geq (1-\alpha) \log f(x)+\alpha \log f(y) + \frac{c}{2} \alpha(1-\alpha) \|x-y\|_2^2$ for all $x,y$ and all $\alpha\in[0,1]$ for some constant $c>0$  In the case of twice-differentiable $h$, this is equivalent to $\nabla^2 (\log f) \preceq -c I$.}
By replacing \prettyref{eq:unimodal} with
	$$f(x) = \begin{cases} 0, & |x| > 1, \\
			1-|x|, &\epsilon<|x|\le 1 \\
			-x^2/(2\epsilon) + 1 - \epsilon /2, &|x|\le \epsilon
		\end{cases} $$
which is strongly log-concave, we again see that $f(x) + f(x-1)$ can have a flat piece. Furthermore, it is possible to construct  infinitely differentiable $f$ (by convolving with a mollifier) with the same property; however, such a density is not analytic. Thus, we ask the question:
\begin{quote}
\it 
Given a convex combination of $k$ analytic densities that are strongly log-concave, how many modes can it have?
\end{quote}
The following result gives an $\Omega(k^2)$ lower bound. Whether this is tight is an open question.

\begin{coro}
\label{cor:logconcave}	
	There exist strongly log-concave analytic densities $f_1,\ldots,f_k$ on $\reals$ and weights $\alpha_1,\ldots,\alpha_k$ such that $\alpha_1 f_1+\ldots+\alpha_k f_k$ has $\Omega(k^2)$ local maxima.
\end{coro}

\begin{proof}
Take the $\pi$ supported on $[-a,a]$ from \prettyref{cor:gaussian-modes}. Partition $[-a,a]$ into $k=4a$ consecutive intervals $I_1,\ldots,I_k$ of length $1/4$. Let $\pi_i$ denote the conditional version of $\pi$ on $I_i$ and set $\alpha_i=\pi(I_i)$. 
Recall the fact that $(\log (\mu * \varphi))'' \geq 1 - b^2$ for any probability measure $\mu$ supported on an interval of length $2b$; 
this follows from the well-known identity $(\log (\mu * \varphi))''(y) = 1 - \Var(X|X+Z=y)$, where $X\sim\mu$ and $Z \sim N(0,1)$ are independent.
Then $f_i \triangleq \pi_i*\varphi$ is strongly log-concave satisfying $(\log f_i)''\geq 3/4$. 
Since $\pi * \varphi= \sum_{i=1}^k \alpha_i f_i$, the desired conclusion then follows from  \prettyref{cor:gaussian-modes}.
\end{proof}

\subsection{Further open problems}
	\label{sec:discuss-open}

	In addition to those on exponential mixtures and constrained NPMLE mentioned in \prettyref{sec:discuss-compact} and \prettyref{sec:discuss-fail}, we end the paper by describing  some further open problems on the structure of NPMLE:

	\subsubsection{Lower bound for NPMLE}
	A particular consequence of \prettyref{thm:gaussian} is the following: when the sample are generated from a finite Gaussian mixture, say, $N(0,1)$, with high probability the NPMLE outputs a Gaussian mixture with at most $O(\log n)$ components. 
	To understand the NPMLE from the perspective of overparameterization, it is of great interest to determine whether this bound is tight.
	(Note that the reasoning in \prettyref{rmk:gaussian-tight} only shows that this is tight when the true density is $N(0,\sigma^2)$ for any $\sigma^2>1$.)
	If so, it would show that the unpenalized NPMLE indeed selects a slightly inflated model (at the price of being fully automatic) and the model selection criterion, such as BIC \cite{leroux1992consistent,keribin2000consistent}, is genuinely needed for achieving consistency in estimating the order of the mixture.
	
	As mentioned in \prettyref{sec:intro}, such lower bound is known to hold for the Grenander estimator (NPMLE for monotone density): 
	if the true density $f$ is uniform, then the number of pieces in the Grenander estimator is asymptotically $N(\log n,\log n)$ \cite{groeneboom1993isotonic}. This is a direct consequence of a celebrated result of Sparre Andersen on the least concave majorant of empirical CDF \cite{andersen1954fluctuations,groeneboom2020}, whose discontinuity in slope correspond to the atoms of Grenander estimator. For the NPMLE in mixture models, no such simple characterization is known other than the first-order condition \prettyref{eq:kkt}.

	%
	%

	\subsubsection{Multivariate models}
	Compared to the univariate case, the structure of the NPMLE is far less well understood for multivariate models. Indeed, the general theory developed in \cite{Lindsay1995}	relies on the parameter space being one-dimensional.	
	For instance, for the simplest Gaussian location mixture, even the uniqueness of the solution is open in dimension $d\geq 2$.
Similar to the analysis in the current paper, bounding the number of atoms in the NPMLE boils down to counting the critical points of a Gaussian mixture \prettyref{eq:F-gaussian} 
with centers being the individual observations, which, if drawn from a subgaussian distribution, lie in a hypercube of size $O(\sqrt{\log n})$ with high probability. The construction in \cite[Proposition 3]{KK20} shows that there exists a mixing distribution on $[-a,a]^d$ whose Gaussian location mixture has $\Omega(a^{2d})$ modes. However, it is unclear whether this is tight and directly extending the complex-analytic technique in this paper to multiple dimensions appears challenging. 

	On the other hand, although the uniqueness of the NPMLE is not settled, the usual analysis of maximal likelihood (zeroth-order optimality) yields statistical guarantees that apply to any solution of the NPMLE \cite{dicker2016high,saha2020nonparametric}. For example, extending the work of \cite{zhang2009generalized}, \cite[Corollary 2.2]{saha2020nonparametric} showed that if the mixing distribution is compactly supported, then the estimated mixture density has squared Hellinger accuracy of $O_d( (\log n)^{d+1}/n)$. 
	
	In view of the above results, we conjecture that the solution to the NPMLE for multivariate Gaussian mixtures is unique and, furthermore, given a subgaussian sample of size $n$ it is typically $(\log n)^{C(d)}$-atomic when the dimension $d$ is not too big.

	\subsubsection{Log-concave NPMLE}
			The NPMLE for log-concave densities is well-studied in nonparametric statistics literature. 
			Basic properties (such as the almost sure existence and uniqueness) and computational algorithms are obtained in \cite{pal2007estimating,dumbgen2009maximum} in one dimension and extended to multiple dimensions \cite{cule2010maximum}. In particular, similar to the NPMLE for monotone density (Grenander estimator) which is piecewise constant, the logarithmic of the NPMLE for log-concave density is piecewise affine with at most $n$ pieces; however, unlike the Grenander estimator, its typical structure (e.g.~the number of pieces) is little understood, partly because the optimal condition is more complicated.


In terms of statistical results, in one dimension the minimax squared Hellinger rate is shown to be $\Theta(n^{-4/5})$ \cite{doss2016global,kim2016global}. 
For dimension $d\geq 2$, \cite{kim2016global} proved the minimax lower bound $\Omega(n^{-2/(d+1)})$ and showed it can be attained by the NPMLE up to logarithmic factors for $d=2$ and $3$. This near-optimality of NPMLE is recently extended to any dimension in \cite{kur2019optimality,han2019global}.
	In view of the corresponding results for the Grenander estimator, if one interprets the minimax rate as the effective dimension divided by the sample size, it is reasonable to conjecture that the typical number of pieces in the log-concave NPMLE is $O(n^{1/5})$ and $O(n^{(d-1)/(d+1)})$ for $d\geq 2$.

\section*{Acknowledgment}
This work was partially completed when the authors were visiting the Information Processing Group at the School of Computer and Communication Sciences of EPFL, 
whose generous support is gratefully acknowledged and whose seminar, canceled due to COVID-19, nevertheless brought the independent work \cite{dysto2020} to our attention.
The authors thank Pengkun Yang for helpful discussions at the onset of the project and for informing us \cite{KK20}.
The authors are also grateful to Roger Koenker for helpful discussion on \cite{koenker2019comment} and providing numerical simulation.

Y.~Wu is supported in part by the NSF Grant CCF-1900507, NSF CAREER award CCF-1651588, and an Alfred Sloan
fellowship. Y.~Polyanskiy is supported in part by the Center for Science of Information (CSoI), an NSF Science and Technology Center, under grant agreement CCF-09-39370, and 
the MIT-IBM Watson AI Lab.



\begin{thebibliography}{DYPS20}

\bibitem[Bir89]{birge1989grenader}
Lucien Birg\'e.
\newblock The {G}renader estimator: A nonasymptotic approach.
\newblock {\em The Annals of Statistics}, pages 1532--1549, 1989.

\bibitem[CSS10]{cule2010maximum}
Madeleine Cule, Richard Samworth, and Michael Stewart.
\newblock Maximum likelihood estimation of a multi-dimensional log-concave
  density.
\newblock {\em Journal of the Royal Statistical Society: Series B (Statistical
  Methodology)}, 72(5):545--607, 2010.

\bibitem[DR09]{dumbgen2009maximum}
Lutz D{\"u}mbgen and Kaspar Rufibach.
\newblock Maximum likelihood estimation of a log-concave density and its
  distribution function: Basic properties and uniform consistency.
\newblock {\em Bernoulli}, 15(1):40--68, 2009.

\bibitem[DW16]{doss2016global}
Charles~R Doss and Jon~A Wellner.
\newblock Global rates of convergence of the {MLE}s of log-concave and
  $s$-concave densities.
\newblock {\em The Annals of Statistics}, 44(3):954, 2016.

\bibitem[DYPS20]{dysto2020}
A.~Dytso, S.~Yagli, H.~V. Poor, and S.~Shamai (Shitz).
\newblock The capacity achieving distribution for the amplitude constrained
  additive {G}aussian channel: An upper bound on the number of mass points.
\newblock {\em IEEE Transactions on Information Theory}, 66(4):2006--2022,
  2020.
\newblock arxiv:1901.03264v4.

\bibitem[DZ16]{dicker2016high}
Lee~H Dicker and Sihai~D Zhao.
\newblock High-dimensional classification via nonparametric empirical bayes and
  maximum likelihood inference.
\newblock {\em Biometrika}, 103(1):21--34, 2016.

\bibitem[Egg58]{eggleston1958convexity}
H.~G. Eggleston.
\newblock {\em Convexity}, volume~47 of {\em Tracts in Math and Math. Phys}.
\newblock Cambridge University Press, 1958.

\bibitem[GG71]{good1971nonparametric}
IJ~Good and Ray~A Gaskins.
\newblock Nonparametric roughness penalties for probability densities.
\newblock {\em Biometrika}, 58(2):255--277, 1971.

\bibitem[GJ14]{groeneboom2014nonparametric}
Piet Groeneboom and Geurt Jongbloed.
\newblock {\em Nonparametric estimation under shape constraints}, volume~38.
\newblock Cambridge University Press, 2014.

\bibitem[GL93]{groeneboom1993isotonic}
Piet Groeneboom and HP~Lopuhaa.
\newblock Isotonic estimators of monotone densities and distribution functions:
  basic facts.
\newblock {\em Statistica Neerlandica}, 47(3):175--183, 1993.

\bibitem[Gre56]{grenander1956theory}
Ulf Grenander.
\newblock On the theory of mortality measurement. {P}art {II}.
\newblock {\em Scandinavian Actuarial Journal}, 1956(2):125--153, 1956.

\bibitem[Gre81]{grenander1981abstract}
Ulf Grenander.
\newblock {\em Abstract inference}.
\newblock John Wiley \& Sons, New York, 1981.

\bibitem[Gro11]{groeneboom2011vertices}
Piet Groeneboom.
\newblock Vertices of the least concave majorant of brownian motion with
  parabolic drift.
\newblock {\em Electronic Journal of Probability}, 16:2334--2358, 2011.

\bibitem[Gro20]{groeneboom2020}
Piet Groeneboom.
\newblock Grenander functionals and {C}auchy's formula.
\newblock {\em Scandinavian Journal of Statistics}, pages 1--20, 2020.
\newblock arXiv preprint arXiv:1902.08806.

\bibitem[GvdV01]{ghosal.vdv}
S.~Ghosal and A.W. van~der Vaart.
\newblock {Entropies and rates of convergence for maximum likelihood and Bayes
  estimation for mixtures of normal densities}.
\newblock {\em The Annals of Statistics}, 29(5):1233--1263, 2001.

\bibitem[GvdV07]{ghosal2007posterior}
Subhashis Ghosal and Aad van~der Vaart.
\newblock Posterior convergence rates of {D}irichlet mixtures at smooth
  densities.
\newblock {\em The Annals of Statistics}, 35(2):697--723, 2007.

\bibitem[GW92]{groeneboom1992information}
Piet Groeneboom and Jon~A Wellner.
\newblock {\em Information bounds and nonparametric maximum likelihood
  estimation}, volume~19.
\newblock Springer Science \& Business Media, 1992.

\bibitem[GW00]{genovese.wasserman}
C.~R. Genovese and L.~Wasserman.
\newblock {Rates of convergence for the Gaussian mixture sieve}.
\newblock {\em Annals of Statistics}, 28(4):1105--1127, 2000.

\bibitem[Han19]{han2019global}
Qiyang Han.
\newblock Global empirical risk minimizers with "shape constraints" are rate
  optimal in general dimensions.
\newblock {\em arXiv preprint arXiv:1905.12823}, 2019.

\bibitem[Jew82]{Jewell1982}
Nicholas~P Jewell.
\newblock Mixtures of exponential distributions.
\newblock {\em The Annals of Statistics}, 10(2):479--484, 1982.

\bibitem[JZ09]{jiang2009general}
Wenhua Jiang and Cun-Hui Zhang.
\newblock General maximum likelihood empirical bayes estimation of normal
  means.
\newblock {\em The Annals of Statistics}, 37(4):1647--1684, 2009.

\bibitem[JZB{\etalchar{+}}16]{jin2016local}
Chi Jin, Yuchen Zhang, Sivaraman Balakrishnan, Martin~J Wainwright, and
  Michael~I Jordan.
\newblock Local maxima in the likelihood of {G}aussian mixture models:
  Structural results and algorithmic consequences.
\newblock In {\em Advances in neural information processing systems}, pages
  4116--4124, 2016.

\bibitem[KDR19]{kur2019optimality}
Gil Kur, Yuval Dagan, and Alexander Rakhlin.
\newblock Optimality of maximum likelihood for log-concave density estimation
  and bounded convex regression.
\newblock {\em arXiv preprint arXiv:1903.05315}, 2019.

\bibitem[Ker00]{keribin2000consistent}
Christine Keribin.
\newblock Consistent estimation of the order of mixture models.
\newblock {\em Sankhy{\=a}: The Indian Journal of Statistics, Series A},
  62(1):49--66, 2000.

\bibitem[KG19]{koenker2019comment}
Roger Koenker and Jiaying Gu.
\newblock Comment: Minimalist $ g $-modeling.
\newblock {\em Statistical Science}, 34(2):209--213, 2019.

\bibitem[Kim14]{kim2014minimax}
Arlene~KH Kim.
\newblock Minimax bounds for estimation of normal mixtures.
\newblock {\em Bernoulli}, 20(4):1802--1818, 2014.

\bibitem[KK20]{KK20}
Navin Kashyap and Manjunath Krishnapur.
\newblock How many modes can a constrained {G}aussian mixture have?
\newblock {\em Arxiv preprint arXiv:2005.01580}, April 2020.

\bibitem[KM14]{koenker2014convex}
Roger Koenker and Ivan Mizera.
\newblock Convex optimization, shape constraints, compound decisions, and
  empirical {B}ayes rules.
\newblock {\em Journal of the American Statistical Association},
  109(506):674--685, 2014.

\bibitem[KS16]{kim2016global}
Arlene~KH Kim and Richard~J Samworth.
\newblock Global rates of convergence in log-concave density estimation.
\newblock {\em The Annals of Statistics}, 44(6):2756--2779, 2016.

\bibitem[KW56]{KW56}
Jack Kiefer and Jacob Wolfowitz.
\newblock Consistency of the maximum likelihood estimator in the presence of
  infinitely many incidental parameters.
\newblock {\em The Annals of Mathematical Statistics}, pages 887--906, 1956.

\bibitem[Lai78]{Laird1978}
Nan Laird.
\newblock Nonparametric maximum likelihood estimation of a mixing distribution.
\newblock {\em Journal of the American Statistical Association},
  73(364):805--811, 1978.

\bibitem[Ler92]{leroux1992consistent}
Brian~G Leroux.
\newblock Consistent estimation of a mixing distribution.
\newblock {\em The Annals of Statistics}, 20(3):1350--1360, 1992.

\bibitem[Lin83a]{lindsay1983geometry1}
Bruce~G Lindsay.
\newblock The geometry of mixture likelihoods: a general theory.
\newblock {\em The Annals of Statistics}, 11(1):86--94, 1983.

\bibitem[Lin83b]{lindsay1983geometry2}
Bruce~G Lindsay.
\newblock The geometry of mixture likelihoods, part {II}: the exponential
  family.
\newblock {\em The Annals of Statistics}, 11(3):783--792, 1983.

\bibitem[Lin95]{Lindsay1995}
Bruce~G Lindsay.
\newblock Mixture models: theory, geometry and applications.
\newblock In {\em NSF-CBMS regional conference series in probability and
  statistics}, pages i--163. JSTOR, 1995.

\bibitem[LR93]{lindsay1993uniqueness}
Bruce~G Lindsay and Kathryn Roeder.
\newblock Uniqueness of estimation and identifiability in mixture models.
\newblock {\em Canadian Journal of Statistics}, 21(2):139--147, 1993.

\bibitem[MM11]{maugis2011non}
Cathy Maugis and Bertrand Michel.
\newblock A non asymptotic penalized criterion for gaussian mixture model
  selection.
\newblock {\em ESAIM: Probability and Statistics}, 15:41--68, 2011.

\bibitem[PW20]{PW20inapprox}
Yury Polyanskiy and Yihong Wu.
\newblock Note on approximating the {L}aplace transform of a {G}aussian on a
  unit disk.
\newblock {\em arXiv preprint arXiv:2008.13372}, 2020.

\bibitem[PWM07]{pal2007estimating}
Jayanta~Kumar Pal, Michael Woodroofe, and Mary Meyer.
\newblock {\em Estimating a Polya frequency function$_2$}, volume Volume 54 of
  {\em Lecture Notes--Monograph Series}, pages 239--249.
\newblock Institute of Mathematical Statistics, Beachwood, Ohio, USA, 2007.

\bibitem[Rob50]{robbins1950generalization}
Herbert Robbins.
\newblock A generalization of the method of maximum likelihood: {E}stimating a
  mixing distribution ({A}bstract).
\newblock In {\em Annals of Mathematical Statistics}, volume~21, pages
  314--315, 1950.

\bibitem[SA54]{andersen1954fluctuations}
Erik Sparre~Andersen.
\newblock On the fluctuations of sums of random variables.
\newblock {\em Mathematica Scandinavica}, pages 263--285, 1954.

\bibitem[SG20]{saha2020nonparametric}
Sujayam Saha and Adityanand Guntuboyina.
\newblock On the nonparametric maximum likelihood estimator for gaussian
  location mixture densities with application to gaussian denoising.
\newblock {\em The Annals of Statistics}, 48(2):738--762, 2020.

\bibitem[Sil82]{silverman1982estimation}
Bernard~W Silverman.
\newblock On the estimation of a probability density function by the maximum
  penalized likelihood method.
\newblock {\em The Annals of Statistics}, pages 795--810, 1982.

\bibitem[Sim76]{simar1976maximum}
L\'{e}opold Simar.
\newblock Maximum likelihood estimation of a compound poisson process.
\newblock {\em The Annals of Statistics}, pages 1200--1209, 1976.

\bibitem[SW14]{saumard2014log}
Adrien Saumard and Jon~A Wellner.
\newblock Log-concavity and strong log-concavity: a review.
\newblock {\em Statistics surveys}, 8:45--114, 2014.

\bibitem[Tij71]{tijdeman71}
R.~Tijdeman.
\newblock On the number of zeros of general exponential polynomials.
\newblock {\em Indagationes Mathematicae (Proceedings)}, 74:1 -- 7, 1971.

\bibitem[vdG00]{vandeGeer2000}
Sara van~de Geer.
\newblock {\em Empirical Processes in {M}-Estimation}.
\newblock Cambridge University Press, 2000.

\bibitem[WV10]{WV2010}
Yihong Wu and Sergio Verd{\'u}.
\newblock The impact of constellation cardinality on {G}aussian channel
  capacity.
\newblock In {\em 2010 48th Annual Allerton Conference on Communication,
  Control, and Computing (Allerton)}, pages 620--628. IEEE, 2010.

\bibitem[WY20]{WY18}
Yihong Wu and Pengkun Yang.
\newblock Optimal estimation of {G}aussian mixtures with denoised method of
  moments.
\newblock {\em The Annals of Statistics}, 48(4):1981--2007, 2020.
\newblock arxiv:1807.07237.

\bibitem[Zha09]{zhang2009generalized}
Cun-Hui Zhang.
\newblock Generalized maximum likelihood estimation of normal mixture
  densities.
\newblock {\em Statistica Sinica}, pages 1297--1318, 2009.

\end{thebibliography}

\newcommand{\etalchar}[1]{$^{#1}$}

\end{document}